\documentclass{amsart}

\allowdisplaybreaks[1]

\usepackage{amssymb}
\usepackage[T1]{fontenc}
\usepackage{hyperref}
\usepackage{amscd}
\usepackage[all]{xy}

\usepackage{multicol}
\title{Lectures on Klein surfaces and their fundamental group}
\author{Florent Schaffhauser}
\address{Departamento de Matem\'aticas, Universidad de Los Andes, Bogot\'a, Colombia.}
\email{florent@uniandes.edu.co}

\newcommand{\GL}{\mathbf{GL}}
\newcommand{\PGL}{\mathbf{PGL}}
\newcommand{\U}{\mathbf{U}}

\newcommand{\PSL}{\mathbf{PSL}}
\newcommand{\bE}{\mathbf{E}}

\newcommand{\R}{\mathbb{R}}
\newcommand{\C}{\mathbb{C}}
\newcommand{\Z}{\mathbb{Z}}

\newcommand{\cH}{\mathcal{H}}

\newcommand{\Hom}{\mathrm{Hom}}
\newcommand{\Aut}{\mathrm{Aut}}

\newcommand{\Id}{\mathrm{Id}}

\newcommand{\Out}{\mathrm{Out}}
\newcommand{\Inn}{\mathrm{Inn}}
\newcommand{\Cov}{\mathrm{Cov}}
\newcommand{\Fib}{\mathrm{Fib}}
\newcommand{\Sets}{\mathrm{Sets}}

\newcommand{\Spm}{\mathrm{Spm}\,}

\renewcommand{\phi}{\varphi}

\renewcommand{\Im}{\mathrm{Im}\,}
\renewcommand{\mod}[1]{\ (\mathrm{mod}\,#1)}

\newcommand{\lra}{\longrightarrow}
\newcommand{\lmt}{\longmapsto}
\newcommand{\RP}{\R\mathbf{P}}
\newcommand{\CP}{\C\mathbf{P}}

\newcommand{\Ga}{\Gamma}
\newcommand{\Si}{\Sigma}
\newcommand{\si}{\sigma}

\newcommand{\piR}{\pi_1^{\R}((X,\si);x)}
\newcommand{\piC}{\pi_1(X;x)}
\newcommand{\piRp}{\pi_1^{\R}((X',\si');x')}
\newcommand{\piCp}{\pi_1(X';x')}
\newcommand{\Xt}{\widetilde{X}}
\newcommand{\op}{\mathrm{op}}
\newcommand{\HomR}{\Hom_{\Z/2\Z}(\piR;\GL^\pm(V))}
\newcommand{\HomC}{\Hom(\piC;\GL(V))}
\newcommand{\HomRunit}{\Hom_{\Z/2\Z}(\piR;\U^\pm(V,h))}
\newcommand{ \HomRunits}{\Hom_{\Z/2\Z}(\piR;\U^\pm(V))}
\newcommand{\HomCunit}{\Hom(\piC;\U(V))}
\newcommand{\sit}{\widetilde{\si}}
\newcommand{\cR}{\mathcal{R}}
\newcommand{\cRsialpha}{\mathcal{R}(\si,\alpha)}

\newcommand{\HomRunitspsd}{\Hom_{\Z/2\Z}(\piR;\U(V)\rtimes_{\alpha}\Z/2\Z)}
\newcommand{\Fix}{\mathrm{Fix}}

\newcommand{\bs}{\backslash}
\newcommand{\ga}{\gamma}
\newcommand{\gat}{\widetilde{\ga}}
\newcommand{\xt}{\widetilde{x}}
\newcommand{\fibre}{p^{-1}(\{x\})}
\newcommand{\fibreq}{q^{-1}(\{x\})}
\newcommand{\car}{\curvearrowright}
\newcommand{\cal}{\curvearrowleft}
\newcommand{\fibrepsi}{p^{-1}(\{\si(x)\})}
\newcommand{\eps}{\varepsilon}
\newcommand{\etat}{\widetilde{\eta}}
\newcommand{\Sit}{\widetilde{\Si}}

\newcommand{\Ad}[1]{\mathrm{Ad}_{#1}}
\newcommand{\ov}[1]{\overline{#1}}

\newtheorem{theorem}{Theorem}[section]
\newtheorem{proposition}[theorem]{Proposition}
\newtheorem{lemma}[theorem]{Lemma}
\newtheorem{corollary}[theorem]{Corollary}

\theoremstyle{definition}
\newtheorem{definition}[theorem]{Definition}
\newtheorem{remark}[theorem]{Remark}
\newtheorem{example}[theorem]{Example}

\newtheorem*{acknowledgments}{Acknowledgments}

\newtheoremstyle{exercise}{\topsep}{5pt}%
     {}
     {}
     {\scshape}
     {.}
     {\topsep}
     {\thmname{#1}\thmnumber{\,#2}\thmnote{\,(#3)}}

\numberwithin{equation}{section}

\theoremstyle{exercise}
\newtheorem{exercise}{Exercise}[section]

\keywords{Klein surfaces, Fundamental groups}

\begin{document}

\begin{abstract}
A Klein surface may be seen as a Riemann surface $X$ endowed with an anti-holomorphic involution $\si$. The fundamental group of the Klein surface $(X,\si)$ is the orbifold fundamental group of $[X/\si]$, the quotient orbifold for the $\Z/2\Z$-action on $X$ defined by $\si$, and its most basic yet most important property is that it contains the topological fundamental group of $X$ as an index two subgroup. The goal of these lectures is to give an introduction to the study of the fundamental group of a Klein surface. We start by reviewing the topological classification of Klein surfaces and by explaining the relation with real algebraic curves. Then we introduce the fundamental group of a Klein surface and present its main basic properties. Finally, we study the variety of unitary representations of this group and relate it to the representation variety of the topological fundamental group of $X$.
\end{abstract}

\maketitle

\tableofcontents

These notes are based on a series of three 1-hour lectures given in 2012 at the CRM in Barcelona, as part of the event \textit{Master Class and Workshop on Representations of Surface Groups}, itself a part of the research program \textit{Geometry and Quantization of Moduli Spaces}. The goal of the lectures was to give an introduction to the general theory of Klein surfaces, particularly the appropriate notion of fundamental groups for such surfaces, emphasizing throughout the analogy with a more algebraic perspective on fundamental groups in real algebraic geometry. Indeed, the topological fundamental group of a space $X$ is usually defined in terms of paths in $X$, as the set of homotopy classes of loops at given base point $x\in X$. Then, under mild topological assumptions on $X$, the universal cover $(\Xt,\xt)$ of $(X,x)$ is constructed and it is shown that the fundamental group of $X$ at $x$ is isomorphic to the automorphism group of this universal cover and that intermediate covers correspond to subgroups of the fundamental group. In \textit{SGA 1} (reprinted in \cite{SGA1}), Grothendieck used this last property of the fundamental group, that it classifies covers of a connected space $X$, as a definition for the algebraic fundamental group of a connected scheme over a field: the algebraic fundamental group is, in this case, a Galois group for an appropriate category of covers of $X$, and it is an augmentation of the absolute Galois group of the base field. He also showed that, in case $X$ is a complex algebraic variety, its algebraic fundamental group is the profinite completion of the topological fundamental group of the associated complex analytic space. Reversing the perspective, one could say that the topological fundamental group of a Riemann surface is the discrete analogue of the algebraic fundamental group of a complex algebraic curve. From that point of view, the goal of these notes is to explain what group is the discrete analogue of the algebraic fundamental group when the complex curve is actually defined over the reals.

The complex analytic way to think about a real algebraic curve is as a pair $(X,\si)$ where $X$ is a Riemann surface and $\si$ is an anti-holomorphic involution of $X$. Following the properties of the algebraic fundamental group exposed in \textit{SGA 1}, the discrete fundamental group of $(X,\si)$ should be a group $\piR$ fitting in the following short exact sequence $$1\lra \piC \lra \piR \lra \mathrm{Gal}(\C/\R) \lra 1,$$ where $\piC$ is the topological fundamental group of $X$. The exact sequence above says that $\piR$ is in fact the orbifold fundamental group of $X/\si$, first defined in terms of orbifold paths by Thurston (see \cite{Scott} or \cite{Ratcliffe} for a presentation of the general theory of orbifolds and Thurston's definition). In these lectures, we give a definition of $\piR$ in terms of covering spaces of the real curve $(X,\si)$.

The first section gives a brief account of the general theory of Real Riemann surfaces and Klein surfaces, including the topological classification of compact, connected, Real Riemann surfaces. In the second section, we define the fundamental group of a Klein surface and show how it classifies real covering spaces of $(X,\si)$. And, in the third and final section, we study linear and unitary representations of the fundamental group of a Klein surface and show how the unitary representation variety of the orbifold fundamental group $\piR$ maps to the unitary representation of the topological fundamental group $\piC$, with its image contained in the fixed-point set of an involution.

\begin{acknowledgments}
It is a pleasure to thank all the participants for their kind attention and stimulating questions. Special thanks go to workshop organizers Steven Bradlow and \'Oscar Garc\'ia Prada, program organizers Luis \'Alvarez-C\'onsul and Ignasi Mundet i Riera, and speakers and participants Olivier Guichard, Jacques Hurtubise, Andr\'es Jaramillo Puentes,  Melissa Liu, Gregor Masbaum, Toni Pantev and Richard Wentworth.
\end{acknowledgments}

\section{Klein surfaces and real algebraic curves}

In this section, we review well-known material from the theory of Klein surfaces. More complete references are provided by the survey papers of Natanzon (\cite{Natanzon_survey}) and Huisman (\cite{Huisman_survey}), and by the book of Alling and Greenleaf (\cite{AG}).

\subsection{Algebraic curves and two-dimensional manifolds}

It is a very familiar idea, since the work of Riemann, that to a non-singular complex algebraic curve there is associated an orientable (in fact, oriented) real surface, i.e.\ a two-dimensional manifold. Conversely, any \textit{compact}, connected, orientable surface admits a structure of complex analytic manifold of dimension one (i.e.\ a Riemann surface structure), with respect to which it embeds onto a complex submanifold of $\CP^3$. By a theorem of Chow, the image of such an embedding is algebraic. It is the goal of this subsection to briefly recall that real algebraic curves also have associated surfaces, whose topology and geometry encode certain algebraic properties of the curve.

We begin with a simple example. Let $P\in\R[x,y]$ be a non-constant polynomial in two variables with real coefficients. Of course, $P$ may also be seen as an element of $\C[x,y]$. We set $$X_P(\C) := \{(u,v)\in \C^2\ |\ P(u,v)=0\}$$ and $$X_P(\R):=\{(u,v)\in \R^2\ |\ P(u,v) =0\}.$$ $X_P(\R)$ is a real plane curve and $X_P(\C)$ is a complex plane curve. The complex affine plane $\C^2$ is endowed with a \textit{real structure} (=anti-holomorphic involution) $$\si: \begin{array}{ccc} \C^2 & \lra & \C^2 \\ (u,v) & \lmt & (\ov{u},\ov{v}) \end{array}$$ and the real affine plane $\R^2\subset \C^2$ is equal to $\Fix(\si)$. Since $P$ has real coefficients, $X_P(\C)$ is $\si$-invariant and one has $$X_P(\R) = \Fix(\si|_{X_P(\C)}).$$ In other words, the real solutions to the equation $P=0$ are exactly the $\si$-invariant complex solutions. Note that $X_P(\R)$ may be empty, while, as a consequence of Hilbert's Nullstellensatz,  $X_P(\C)$ is always non-empty, the point being that there is an immediate gain in considering the pair $(X_P(\C),\si)$ in place of $X_P(\R)$~: one is certain never to be talking about the empty set. But there is more underlying structure to this story. Assume that the partial derivatives $\frac{\partial P}{\partial u}$ and $\frac{\partial P}{\partial v}$ do not simultaneously vanish at points of $X_P(\C)$. Then $X_P(\C)$, in the topology induced by the usual topology of $\C^2$, admits a structure of Riemann surface (=a holomorphic atlas locally modeled on open sets of $\C$; see for instance \cite{Donaldson_RS}), and the real structure $\si:(u,v) \lmt (\ov{u},\ov{v})$ of $\C^2$ induces an anti-holomorphic involution of $X_P(\C)$, whose fixed-point set is $X_P(\R)$. The previous example motivates the following definition.

\begin{definition}[Real Riemann surface]\label{def:real_curve}
A \textbf{Real Riemann surface} is a pair $(\Si,\tau)$ where $\Si$ is a Riemann surface and $\tau$ is an anti-holomorphic involution of $\Si$~:
\begin{itemize}
\item $\forall x\in\Si$, the tangent map $T_x\tau: T_x\Si \lra T_{\tau(x)}\Si$ is a $\C$-anti-linear map (in particular $\tau\neq\Id_\Si$).
\item $\tau^2=\Id_\Si$.
\end{itemize}
\noindent We shall call such a $\tau$ a \textbf{real structure} on $\Si$.

A \textbf{homomorphism of Real Riemann surfaces} $$f:(\Si,\tau) \lra (\Si',\tau')$$ is a holomorphic map $f:\Si\lra \Si'$ such that $f\circ\tau = \tau'\circ f$.
\end{definition}

A Real Riemann surface $(\Si,\tau)$ is called connected (resp. compact) if $\Si$ is connected (resp. compact). A Real Riemann surface is an example of a \textit{Real space} in the sense of Atiyah (\cite{Atiyah_reality}). That is why we write the word \textit{Real} with a capital \textit{R}, following Atiyah's convention. We will often say real algebraic curve or real curve in place of Real Riemann surface.

The next example shows that, given a real curve $(\Si,\tau)$, the quotient surface $\Si/\tau$ has an algebraic significance (we denote by $\Si/\tau$ the quotient of $\Si$ by the $\Z/2\Z$-action defined by $\tau$). Let $V_\R := \Spm\R[T]$ be the set of maximal ideals of the ring $\R[T]$, and denote by $V_\C:= \Spm\C[T]$.  Maximal ideals of $\C[T]$ (resp. $\R[T]$) are those generated by irreducible polynomials, so $$V_\C=\{(T-z) : z\in\C\}$$ and $$V_\R = \{(T-a) : a \in\R\} \cup \{(T^2+bT+c) : b,c\in\R\ |\ b^2-4c<0\}.$$ In particular, there is a bijection $V_\C\simeq \C$ and the real structure $\si:z\lmt \ov{z}$ of $\C$ induces an involution $(T-z)\lmt (T-\ov{z})$ of $V_\C$, whose fixed-point set is $$\{(T-a) : a\in\R\}\simeq \R = \Fix(\si) \subset \C.$$ If $\Im{z}\neq 0$, the set $\{z;\ov{z}\}$ is the set of roots of the real polynomial $T^2-(z+\ov{z})T+z\ov{z}$, whose discriminant is $$b^2-4c = (z+\ov{z})^2 - 4|z|^2 = (z-\ov{z})^2 = -4(\Im{z})^2 <0.$$ Conversely, a real polynomial of the form $T^2+bT+c$ satisfying $b^2-4c<0$ has a unique complex root $z$ such that $\Im{z}>0$, the other root being $\ov{z}$. In other words, there is a bijection $$V_\R\simeq \big\{\mathrm{orbits\ of}\ \{\Id_\C;\si\}\simeq\mathrm{Gal}(\C/\R)\ \mathrm{in}\ V_\C\big\}=V_\C/\si.$$ In particular, as a set, $$V_\C/\si= \C/\si \simeq \{z\in\C\,\, |\ \Im{z}\geq 0\}=:\C_+$$ is the "closed" upper half-plane. Thus, the algebraic object $V_\R=\Spm\R[T]$ is naturally in bijection with a surface with boundary, namely the closed upper half-plane $\C/\si$. This suggests considering general quotients of the form $\Si/\tau$, where $(\Si,\tau)$ is a Real Riemann surface in the sense of Definition \ref{def:real_curve}. 

By construction, $\Si/\tau$ is a two-dimensional manifold with boundary. It may be non-orientable and the boundary may be empty. The first thing to observe is that there is a good notion of function on $\Si/\tau$. Indeed, if we denote by $$p:\Si\lra \Si/\tau$$ the canonical projection and $O_\Si$ the sheaf of holomorphic functions on $\Si$, then the group $\Z/2\Z=\{\Id_\Si;\tau\}$ acts on the direct image $p_*O_\Si$ by \begin{equation}\label{real_holomorphic_functions} (\tau\cdot f)(x) := \ov{f\big(\tau(x)\big)}.\end{equation} This is well-defined because a local section $f$ of $p_*O_\Si$ is, by definition, a holomorphic function defined on a $\tau$-invariant open subset of $\Si$, and then so is $\tau\cdot f$ defined as above. Let us denote by $(p_*O_\Si)^\tau$ the subsheaf of $\tau$-invariant local sections of $p_*O_\Si$. Then, any local section of $(p_*O_\Si)^\tau$ defines a continuous complex-valued function on $\Si/\tau$~: \begin{equation}\label{functions_on_Klein_surface} \widehat{f}: \begin{array}{ccc} V:= U/\tau & \lra & \C \\ \left[z\right] & \lmt & \psi\circ f(z) \end{array}\end{equation} where $U$ is a $\tau$-invariant open subset of $\Si$ and $\psi$ is the \textit{folding map} \begin{equation}\label{folding_map}\psi:\begin{array}{ccc} \C & \lra & \C_+ \\ x+iy & \lmt & x+i|y| \end{array}\end{equation} (we recall that we identify $\C/\si$ with the closed upper half-plane $\{z\in\C\ |\ \Im{z}\geq 0\}\subset \C$). It is easy to check that \eqref{functions_on_Klein_surface} is well-defined on $U/\tau$ (for, if $f\in (p_*O_\Si)^\tau$, then $f(\tau(z)) = \ov{f(z)}$, so $\psi\circ f(\tau(z))= \psi\circ f(z)$).

\begin{proposition}
Given a Real Riemann surface $(\Si,\tau)$, the canonical projection $p:\Si \lra \Si/\tau$ induces a diffeomorphism between $\Fix(\tau)\subset \Si$ and the boundary of $\Si/\tau$. Any local section of $(p_*O_\Si)^\tau$ is $\R$-valued on $\partial(\Si/\tau)$ and $(p_*O_\Si)^\tau$ is a subsheaf of the sheaf of continuous functions on $\Si/\tau$.
\end{proposition}

\begin{proof}
A local section of $(p_*O_\Si)^\tau$ is a holomorphic function $f:U\lra \C$ defined on a $\tau$-invariant open subset $U$ of $\Si$ and satisfying, for all $x\in U$, $$\ov{f\big(\tau(x)\big)} = f(x).$$ In particular, if $\tau(x)=x$, then $f(x)\in\R$.
\end{proof}

\noindent As a simple example of a Real Riemann surface, consider the Riemann sphere $\CP^1=\C\cup\{\infty\}$ with real structure $\tau:z\lmt \ov{z}$, whose fixed-point set is the circle $\RP^1=\R\cup\{\infty\}$. Another real structure on $\CP^1$ is given by $\tau_0:z\lmt -\frac{1}{\ov{z}}$, which has no fixed points. These two real structures are the only possible ones on $\CP^1$ (see Exercise \ref{real_structures_on_proj_line}). Examples of real structures on compact curves of genus $1$ are given in Exercise \ref{real_curves_of_genus_one}.

\subsection{Topological types of real curves}

The goal of this subsection is to provide topological intuition on real algebraic curves $(\Si,\tau)$, as well as a very useful classification result. This begins with a topological understanding of the fixed-point set $\Si^\tau$ (the topological surface $\Si/\tau$ will also play a r\^ole later on). We assume throughout that $\Si$ is compact and connected and we denote by $g$ its genus.

\begin{proposition}\label{topology_of_fixed_point_set}
Let $(\Si,\tau)$ be a compact connected Real Riemann surface. Then the fixed-point set $\Si^\tau:=\Fix(\tau)$ is a finite union of circles.
\end{proposition}

\begin{proof}
Let $x\in \Si$ and let $(U,\phi)$ be a $\tau$-invariant holomorphic chart about $x$. In the chart, the map $z\lmt \ov{\tau(z)}$ is a bijective holomorphic map, so it is locally holomorphically conjugate to $z\lmt z$ (a general holomorphic map is locally holomorphically conjugate to $z\mapsto z^k$, see for instance \cite{Szamuely}; it is bijective if and only if $k=1$). So $\tau$ is locally holomorphically conjugate to $z\longmapsto \ov{z}$. In particular, $\Si^\tau=\Fix(\tau)$ is a $1$-dimensional real (analytic) manifold. Since $\Si$ is compact, $\Si^\tau$ is compact and has a finite number of connected components, each of which is a compact connected $1$-dimensional real manifold, therefore diffeomorphic to a circle.
\end{proof}

\begin{definition}
A fixed point of $\tau$ in $\Si$ is called a \textbf{real point} of $(\Si,\tau)$ and a connected component of $\Si^\tau$ is called a circle of real points.
\end{definition}

The first step into the topological classification of real curves is provided by Harnack's inequality\footnote{As a matter of fact, Harnack proved Theorem \ref{Harnack_thm} for real \textit{plane} curves and Klein was the one who proved it for arbitrary real curves, as was pointed out to the author by Erwan Brugall\'e.}, which gives a topological upper bound on the number of connected components of $\Si^\tau$ (topological in the sense that it does not depend on the choice of the complex structure on $\Si$, only on its genus). We give a differential-geometric proof which uses the classification of surfaces (which one can read about for instance in \cite{Donaldson_RS}). 

\begin{theorem}[Harnack's inequality, \cite{Harnack}]\label{Harnack_thm}
Let $(\Si,\tau)$ be a Real Riemann surface and let $k$ be the number of connected components of $\Si^\tau$. Then $k\leq g+1.$ If $k=g+1$, then $\Si/\tau$ is homeomorphic to a sphere minus $g+1$ discs.
\end{theorem}

\noindent It should be noted that, in full strength, Theorem \ref{Harnack_thm} says that the above bound is optimal~: for any $g\geq 0$, there exists a Real Riemann surface $(\Si,\tau)$ such that $\Si^\tau$ has exactly $(g+1)$ connected components.

\begin{lemma}\label{Euler_charac_of_quotient}
The Euler characteristic of $\Si$ and $\Si/\tau$ are related in the following way~: $$\chi(\Si/\tau) = \frac{1}{2}\chi(\Si) = 1-g.$$
\end{lemma}

\begin{proof}
The canonical projection $p: \Si \lra \Si/\tau$ is two-to-one outside $\partial(\Si/\tau)$ and induces by restriction a diffeomorphism $\Si^\tau \simeq \partial(\Si/\tau)$. Starting from a triangulation $T$ of $\Si/\tau$, the pullback under $p$ of any triangle not intersecting $\partial(\Si/\tau)$ (call this type (a), say) gives two disjoint triangles in $\Si$ (provided the triangles in the original triangulation $T$ are small enough, which can be assumed). The pullback of a triangle having a vertex on $\partial(\Si/\tau)$ (call this type (b)) gives two triangles in $\Si$ having exactly one vertex in common. And the pullback of a triangle having an edge on $\partial(\Si/\tau)$ (call this type (c)) gives two triangles in $\Si$ having exactly one edge in common. Since $\partial(\Si/\tau)$ is a union of $k$ circles, its Euler characteristic is zero. So, in the triangulation $T$, there are as many triangles of type (b) as those of type (c). In particular, in the triangulation $T$, the number $v_1$ of vertices lying  in $\partial(\Si/\tau)$ is equal to the number $e_1$ of edges lying in $\partial(\Si/\tau)$. We denote by $v_2$ (resp. $e_2$, resp. $c$) the number of remaining vertices (resp. edges, resp. faces) in $T$. Then $$\chi(\Si/\tau) = (v_1+v_2) - (e_1+e_2) + c = v_2 - e_2 + c$$ and, using the pullback triangulation $p^*T$ of $\Si$ described above, $$\chi(\Si)= (v_1 + 2v_2) - (e_1+2 e_2) +2c = 2v_2 - 2e_2 +2c = 2 \chi(\Si/\tau). \qedhere$$
\end{proof}

\begin{proof}[Proof of Theorem \ref{Harnack_thm}]
$\Si/\tau$ is a compact connected surface with $k$ boundary components. It follows from the classification of surfaces that $$\chi(\Si/\tau)\leq 2-k.$$ Indeed, if $\Si/\tau$ is orientable, it is either a sphere minus $k$ discs (in which case $\chi(\Si/\tau)=2-k$) or a connected sum of $\widehat{g}$ tori minus $k$ discs (in which case $\chi(\Si/\tau) = 2-2\widehat{g}-k<2-k$). If $\Si/\tau$ is non-orientable, it is a connected sum of $h\geq 1$ real projective planes $\RP^2$ minus $k$ discs (in which case $\chi(\Si/\tau) = 2-h-k <2-k$). Consequently, using Lemma \ref{Euler_charac_of_quotient}, we have $$2-k \geq \chi(\Si/\tau) = 1-g $$ so $k\leq g+1$. Moreover, if $k=g+1$ then $\chi(\Si/\tau)=2-k$ and we have seen in the course of the proof that this happens only if $\Si/\tau$ is homeomorphic to a sphere minus $k$ discs.
\end{proof}

It is clear from the definition of a homomorphism of real curves that the genus $g$ and the number $k$ of connected components of $\Si^\tau$ are topological invariants of $(\Si,\tau)$. Indeed, if $f:\Si\lra \Si'$ is a homeomorphism such that $f\circ\tau = \tau'\circ f$, then $f(\Si^\tau)\subset (\Si')^{\tau'}$ and $f$ induces a bijection connected components of $\Si^\tau$ and those of $(\Si')^{\tau'}$. Let us now set \begin{equation}\label{orientability_index} a= \left\{ \begin{array}{cl} 0 & \mathrm{if}\ \Si\setminus \Si^\tau\ \mathrm{is\ not\ connected,} \\  1 & \mathrm{if}\ \Si\setminus \Si^\tau\ \mathrm{is\ connected}. \end{array}\right.\end{equation} In particular, if $k=0$, then $a=1$. The real structure $\tau$ is called \textit{dividing} if $a=0$ and \textit{non-dividing} if $a=1$. One sometimes finds the opposite convention in the literature but we do not know of a reason to prefer one or the other.

The rest of this subsection is devoted to showing that $(g,k,a)$ are complete topological invariants of real curves. The main tool to prove this classification result is the notion of \textit{double} of any given compact connected surface $S$.

\begin{definition}\label{double}
Let $S$ be a compact connected surface. A \textbf{double} for $S$ is a triple $(P,\si,q)$ where $P$ is an orientable surface with empty boundary, $\si$ is an orientation-reversing involution of $P$ and $q:P\lra S$ is a continuous map inducing a homeomorphism $\ov{q}: P/\si \overset{\simeq}{\lra} S$. A homomorphism of doubles is a continous map $\phi:P\lra P'$ such that $\phi\circ\si = \si'\circ \phi$.
\end{definition}

\noindent Note that the notion of orientation-reversing self-diffeomorphism makes sense on $P$ even without the choice of an orientation. We warn the reader that the terminology of Definition \ref{double} is not entirely standard (compare \cite{Natanzon_survey,AG}). Whatever the terminology, the point is to have the surface $P$ be \textit{orientable and with empty boundary} and $\si$ be orientation-reversing. As a matter of fact, we explicitly want the surface $P$ to be constructed from $S$ as follows.

\textit{First case~: $S$ is orientable}. Then $S$ is homeomorphic either to a genus $\widehat{g}=0$ surface minus $r$ discs ($r\geq 0$) or to a connected sum of $\widehat{g}$ tori minus $r$ discs ($\widehat{g}\geq 1$, $r\geq 0$). We denote by $\Si_{\widehat{g},r}$ these surfaces. Let $\Si_{\widehat{g},r}^{(1)}$ and $\Si_{\widehat{g},r}^{(2)}$ be two copies of $\Si_{\widehat{g},r}$ and set $$P:= \Big(\Si_{\widehat{g},r}^{(1)} \sqcup \Si_{\widehat{g},r}^{(2)}\Big) \big/ \sim$$ where $x\in \partial\Si_{\widehat{g},l+m}^{(1)}$ is glued to $x\in \partial\Si_{\widehat{g},l+m}^{(2)}$. In particular, if $r\geq 1$, $P$ is a closed orientable surface of genus $2\widehat{g}+r-1$. An involution $\si$ is defined by sending $x\in\Si_{\widehat{g},r}^{(1)}$ to $x\in\Si_{\widehat{g},r}^{(2)}$. Because of the way the gluing is made, $\si$ is orientation-reversing. The fixed-point set of $\si$ in $P$ has $r$ connected components and the map $q:P\lra \Si_{\widehat{g},r}$ sending $x\in\Si_{\widehat{g},r}^{(i)}$ to $x\in\Si_{\widehat{g},r}$ is well-defined and induces a homeomorphism $P/\si\simeq \Si_{\widehat{g},r}$. Note that $P$ is connected if and only if $r\geq 1$, i.e.\ if the original orientable surface $S$ has a non-empty boundary. This is the case of interest to us.

\textit{Second case~: $S$ is non-orientable} Then $S$ contains an open M\"obius band and there exists a unique integer $m=1$ or $2$ and $m$ disjoint open M\"obius bands $U_1,U_m$ contained in $S$ such that the compact connected surface $$S':= S \setminus (U_1 \cup U_m)$$ is orientable (see \cite{Massey} or \cite{Donaldson_RS}). If $S$ has, say, $l$ boundary components ($l\geq 0$), then $S' \simeq \Si_{\widehat{g},l+m}$ for some $\widehat{g}\geq 0$. We denote by $$\partial \Si_{\widehat{g},l+m} = A_1 \sqcup \cdots A_l \sqcup B_1 \sqcup B_m.$$ Each $B_j$ is a circle hence carries a fixed-point free involution $\si_j$ (think of  $z\lmt -z$ on $S^1$). We form $$P :=\Big( \Si_{\widehat{g},l+m}^{(1)} \sqcup \Si_{\widehat{g},l+m}^{(2)} \Big) \big/ \sim$$ where we identify a point $x$  of $A_i^{(1)}$ with $x\in A_i^{(2)}$, and a point $x\in B_j^{(1)}$ with the point $\si_j(x)\in B_j^{(2)}$. In particular, $P$ is a \textit{connected} closed orientable surface of genus $2\widehat{g}+(l+m)-1$. An orientation-reversing involution $\si$ is defined by sending $x\in\Si_{\widehat{g},r}^{(1)}$ to $x\in\Si_{\widehat{g},r}^{(2)}$. Because of the involutions $\si_j$ used in defining the gluing, the fixed-point set of $\si$ has $l$ (not $l+m$) connected components (namely the circles $A_i$) and $\si$ restricted to the circle $B_j$ is $\si_j$. In particular, $P/\si$ is homeomorphic to $S$ (not to $S'$).

Given a compact connected surface $S$, we denote by $(P_S,\si_S)$ the double of $S$ whose construction is given above. Note that to each $(P_S,\tau_S)$ there is associated a triple $(g_S,k_S,a_S)$, where $g_S$ is the genus of $P$, $k_S$ is the number of connected components of $\Fix(\si_S)$ (=the number of connected components of $\partial S$) and $a_S=0$ or $1$ is defined as in \eqref{orientability_index} depending on whether $P_S\setminus P_S^{\si_S}$ is connected. The following observations are direct consequences of the construction of $(P_S,\si_S)$ for a given $S$~:

\begin{itemize}
\item $S$ is non-orientable if and only if $P_S\setminus P_S^{\si_S}$ is connected (i.e.\ $a_S=1$).
\item Two compact connected surfaces $S$ and $S'$ are homeomorphic if and only if there is a homeomorphism $\phi:P_S\lra P_{S'}$ such that $\phi\circ \si_{S} = \si_{S'}\circ\phi$.
\item $S$ is homeomorphic to $S'$ if and only if $(g_S,k_S,a_S)=(g_{S'},k_{S'},a_{S'})$.
\end{itemize}

\noindent As a consequence, taking into account the classification of surfaces, one has~: if $(P,\si)$ is a pair consisting of a closed orientable surface and an orientation-reversing involution and if $S:=P/\si$, then there is a homeomorphism $\phi:P_S\lra P$ such that $\phi\circ\si_S = \si\circ P$. We have thus almost proved the following result, due to Weichold, which gives the topological classification of Real Riemann surfaces.

\begin{theorem}[Weichold's theorem, \cite{Weichold}]\label{Weichold_thm}
Let $(\Si,\tau)$ and $(\Si',\tau')$ be two compact connected Real Riemann surfaces. Then the following conditions are equivalent~:
\begin{enumerate}
\item There exists a homeomorphism $\phi:\Si\lra \Si'$ such that $\phi\circ\tau = \tau'\circ\phi$.
\item $\Si/\tau$ is homeomorphic to $\Si'/\tau'$.
\item $(g,k,a)=(g',k',a')$.
\end{enumerate}

\noindent The triple $(g,k,a)$ is called the \textbf{topological type} of the Real Riemann surface $(\Si,\tau)$. It is subject to the following conditions~:

\begin{enumerate}
\item if $a=0$, then $1\leq k \leq g+1$ and $k\equiv (g+1) \mod{2}$,
\item if $a=1$, then $0\leq k \leq g$.
\end{enumerate}

\end{theorem}

Note that, in full strength, Weichold's theorem says that for each triple $(g,k,a)$ satisfying the conditions above, there actually exists a real algebraic curve with topological type $(g,k,a)$.

\begin{proof}[Proof of Theorem \ref{Weichold_thm}]
There only remains to prove that $(g,k,a)$ must satisfy the conditions listed in the theorem. Note that $0\leq k \leq g+1$ by Harnack's Theorem \ref{Harnack_thm} and that if $k=0$, then $a=1$ by definition of $a$ (see \eqref{orientability_index}). If $a=0$, then $\Si/\tau$ is orientable and homeomorphic to $\Si_{\widehat{g},k}$ for some $\widehat{g}\geq 0$, so $$\chi(\Si/\tau) = 2 - 2\widehat{g} -k.$$ But, by Lemma \ref{Euler_charac_of_quotient}, $$\chi(\Si/\tau) = \frac{1}{2} \chi(\Si) = 1-g,$$ so $k=g+1-2\widehat{g} \equiv g+1 \mod{2}$. If $a=1$, then $\Si/\tau$ is non-orientable and is therefore a connected sum of $h\geq 1$ real projective planes $\RP^2$, minus $k$ discs. So $\chi(\Si/\tau)=2-h-k$ and again $\chi(\Si/\tau)=1-g$, so $k-g=1-h\leq 0$.
\end{proof}

\subsection{Dianalytic structures on surfaces}

In this subsection, a surface will be a $2$-dimensional topological manifold (in particular, Hausdorff and second countable), possibly with boundary. It is well-known that a compact, connected, orientable surface admits a structure of complex analytic manifold of dimension one (if the surface has a boundary, the statement should be taken to mean that it is a compact submanifold with boundary inside an open Riemann surface, as we shall see later). The goal of the present subsection is to recall what type of structure one may put on the surface when it is not necessarily orientable: then it admits a structure of complex \textit{dianalytic} manifold of dimension one, also known as a Klein surface. The reference text for the general  theory of Klein surface is the book by Alling and Greenleaf (\cite{AG}). Dianalytic structures have recently found applications in open Gromov-Witten theory (see \cite{Liu}). In what follows, we use the terms 'holomorphic' and 'analytic' interchangeably. We recall that, if $U$ is an open subset of $\C$, a function $f:U\lra \C$ is called analytic if $\ov{\partial}f=0$ in $U$ and anti-analytic if $\partial f=0$.

\begin{definition}[Dianalytic function, open case]
Let $U$ be an open subset of $\C$. A function $f:U\lra \C$ is called \textbf{dianalytic} in $U$ if its restriction $f|_V$ to any given component $V$ of $U$ satisfies either $\ov{\partial}(f|_V) = 0$ or $\partial(f|_V)=0$, i.e.\,\,$f$ restricted to such a connected component is either analytic or anti-analytic.
\end{definition}

\noindent In order to put atlases with analytic or dianalytic transition maps on surfaces of the generality considered in this subsection, it is necessary to extend the notion of dianalyticity to the boundary case.

\begin{definition}[Analytic and dianalytic functions, boundary case]
Let $A$ be an open subset of  the closed upper half-plane $\C_+:=\{z\in\C\ |\ \Im{z}\geq 0\}$. A function $f:A\lra \C_+$ is called analytic (resp\,\,dianalytic) on $A$ if there exists an open subset $U$ of $\C$ such that $U\supset A$ and $f$ is the restriction to $A$ of an analytic (resp\,\,dianalytic) function on $U$.
\end{definition}

\noindent The next result will be key when relating Klein surfaces (Definition \ref{def_Klein_surface}) and Real Riemann surfaces (Definition \ref{def:real_curve}).

\begin{theorem}[The Schwarz reflection principle, e.g.\,\cite{Conway}]\label{Schwarz}
Let $A$ be an open subset of $\C_+=\{z\in\C\ |\ \Im{z}\geq 0\}$ and let $f:A\lra\C_+$ be a function satisfying~:
\begin{enumerate}
\item $f$ is continuous on $A$,
\item $f$ is analytic in the interior of $A$,
\item $f(A\cap\R)\subset \R$.
\end{enumerate}
\noindent Then there exists a unique analytic function $$F:A\cup \ov{A} \lra \C,$$ defined on the open subset $A\cup\ov{A}$ of $\C$, such that~:
\begin{enumerate}
\item $F|_A=f$,
\item $\forall\, z\in A\cup\ov{A},\, F(\ov{z})=\ov{F(z)}.$
\end{enumerate}
\end{theorem}

\noindent We can now define Klein surfaces as objects that naturally generalize Riemann surfaces. In particular, this will handle the case of Riemann surfaces with boundary.

\begin{definition}[Klein surface]\label{def_Klein_surface}
A \textbf{Klein surface} is a surface $\Si$ equipped with an atlas $(U_i,\phi_i)$ satisfying~:
\begin{enumerate}
\item $\phi_i:U_i\lra V_i$ is a homeomorphism onto an open subset $V_i$ of $\C_+=\{z\in\C\ |\ \Im{z}\geq 0\}$,
\item if $U_i\cap U_j\neq\emptyset$, the map $$\phi_i\circ\phi_j^{-1}: \phi_j(U_i\cap U_j) \lra \phi_i(U_i\cap U_j)$$ is dianalytic.
\end{enumerate}
\noindent Such an atlas will be called a dianalytic atlas.\\
A \textbf{Riemann surface with boundary} is a surface $\Si$ equipped with an atlas $(U_i,\phi_i)$ satisfying~:
\begin{enumerate}
\item $\phi_i:U_i\lra V_i$ is a homeomorphism onto an open subset $V_i$ of $\C_+=\{z\in\C\ |\ \Im{z}\geq 0\}$,
\item if $U_i\cap U_j\neq\emptyset$, the map $$\phi_i\circ\phi_j^{-1}: \phi_j(U_i\cap U_j) \lra \phi_i(U_i\cap U_j)$$ is analytic.
\end{enumerate}
Such an atlas will be called an analytic atlas. It can only exist on an \textit{orientable} surface.
\end{definition}

\noindent So a Riemann surface in the usual sense is a Riemann surface with boundary and a Riemann surface with boundary is a Klein surface with a dianalytic atlas which is in fact analytic. As an example, any open subset $A$ of $\C_+$ is a Riemann surface with (possibly empty) boundary. More interesting examples are provided by the following result.

\begin{theorem}[\cite{AG}]
Let $S$ be a compact connected topological surface, possibly non-orientable and with boundary. Then there exists a dianalytic atlas on $S$. If $S$ is orientable, there exists an analytic atlas on it.
\end{theorem}

\noindent The notion of homomorphism of Klein surfaces is somewhat more intricate.

\begin{definition}
A \textbf{homomorphism} between two Klein surfaces $\Si$ and $\Si'$ is a continuous map $f:\Si\lra 
\Si'$ such that~:
\begin{enumerate}
\item $f(\partial\Si)\subset \partial\Si'$,
\item for any $x\in \Si$, there exist \textit{analytic} charts $\phi:U\overset{\simeq}{\lra} V$ and $\phi':U'\overset{\simeq}{\lra} V'$, respectively about $x\in\Si$ and $f(x)\in\Si'$, and an \textit{analytic} map $F:\phi(U)\lra \C$ such that the following diagram, involving the folding map $$\psi:\begin{array}{ccc} \C & \lra & \C_+ \\ x+iy & \lmt & x+i|y| \end{array}$$ is commutative~:
$$
\xymatrix{
U \ar[rr]^{f} \ar[d]^{\phi} &  & U'  \ar[d]^{\phi'}\\
V \ar[r]^{F} & \C \ar[r]^{\psi} & V'.
}
$$
\end{enumerate}
\end{definition}

\noindent As an example, the Schwarz reflection principle implies that a dianalytic function $f:A\lra \C_+$  which is real on the boundary of $A$ is a homomorphism of Klein surfaces (Exercise \ref{example_of_morphism}). The next result relates Klein surfaces to Real Riemann surfaces.

\begin{theorem}[\cite{AG}]\label{Klein_vs_real}
Let $\Si$ be a Klein surface. Then there exist a Real Riemann surface $(\Sit,\tau)$ and a homomorphism of Klein surfaces $p:\Sit\lra \Si$ such that $p\circ \tau=p$ and the induced map $\bar{p}: \Sit/\tau\overset{\simeq}{\lra} \Si$ is an isomorphism of Klein surfaces. The triple $(\Sit,\tau,p)$ is called an \textbf{analytic double} of $\Si$.\\ If $(\Sit',\tau',p')$ is another analytic double of $\Si$, then there exists an analytic isomorphism $f:\Si\lra\Sit$ such that $f\circ\tau=\tau'\circ f$, i.e\,\,$(\Si,\tau)$ and $(\Si',\tau')$ are isomorphic as Real Riemann surfaces, and this construction defines an equivalence of categories between the category of Klein surfaces and the category of Real Riemann surfaces.
\end{theorem}

\noindent The last statement means that, given a Real Riemann surface $(\Si,\tau)$, the quotient $\Si/\tau$ is a Klein surface with analytic double $(\Si,\tau)$ and that homomorphisms between two Klein surfaces $\Si_1$ and $\Si_2$ correspond bijectively to homomorphisms of Real Riemann surfaces between their analytic doubles. Note that an analytic double of a Klein surface is a double in the sense of Definition \ref{double} (we recall that Real Riemann surfaces in these notes have empty boundary). If the Klein surface $\Si$ is in fact a Riemann surface (without boundary), its analytic double is $\Sit=\Si \sqcup \ov{\Si}$, where $\ov{\Si}$ denotes the surface $\Si$ where holomorphic functions have been replaced by anti-holomorphic functions (in particular, the identity map $\Si\lra \ov{\Si}$ of the underlying sets is anti-holomorphic and therefore defines a real structure on $\Sit$). Other notions of doubles are presented in \cite{AG} (the orienting double and the Schottky double).\\
It is not so difficult to sketch a proof of the existence of an analytic double of a given Klein surface $\Si$. If $(\phi_i: U_i \overset{\simeq}{\lra} V_i)_{i\in I}$ is a dianalytic atlas for $\Si$, form the open sets $W_i=V_i \cup \ov{V_i}$ in $\C$ and consider the topological space $\Omega:=\sqcup_{i\in I}W_i$. For simplicity, we assume that the $V_i$ are all discs or half-discs in $\C_+$. If $\phi_i\circ\phi_j^{-1}$ is analytic, we glue $\phi_j(U_i\cap U_j) \subset V_j$ to $\phi_i(U_i\cap U_j) \subset V_i$ and $\ov{\phi_j}(U_i\cap U_j)\subset\ov{V_j}$ to $\ov{\phi_i}(U_i\cap U_j) \subset \ov{V_i}$. If $\phi_i\circ\phi_j^{-1}$ is anti-analytic, we glue $\phi_j(U_i\cap U_j) \subset V_j$ to $\ov{\phi_i}(U_i\cap U_j) \subset \ov{V_i}$ and $\ov{\phi_j}(U_i\cap U_j)\subset\ov{V_j}$ to $\phi_i(U_i\cap U_j) \subset V_i$. This defines an equivalence relation $\sim$ on $\Omega$ and we set $\Sit:=\Omega/\sim$, with the quotient topology. We let $[W_i]$ denote the image of $W_i\subset \Omega$ under the canonical projection to $\Sit$. There is a canonical homeomorphism $z_i:[W_i]\lra W_i$ and we note that $\Sit$ has empty boundary. By definition of the gluing on $\Omega$, the map $z_i\circ z_j^{-1}$ is always analytic~: for instance, if $\phi_i\circ\phi_j^{-1}$ is anti-analytic, then $z_i\circ z_j^{-1}$ is the map $$v \lmt \left\{ \begin{array}{cl} \ov{\phi_i\circ\phi_j^{-1}} (v) & \mathrm{if}\ v\in V_j \\ \phi_i\circ\phi_j^{-1} (v) & \mathrm{if}\ v\in \ov{V_j}\end{array}\right.$$ so it is analytic by virtue of the Schwarz reflection principle (which is only used, really, in the case where $V_i\cap V_j$ intersects $\R$ in $\C_+$). The real structure $\tau$ on $\Sit$ is just complex conjugations in the local charts $([W_i],z_i)$ and, if read in local analytic charts, the projection $p:\Sit \lra \Si$ takes $x=a+\sqrt{-1}\,b \in W_i$ to $a+\sqrt{-1}\,|b|$, so it is a homomorphism of Klein surfaces. We refer to \cite{AG} for the rest of the proof of Theorem \ref{Klein_vs_real}.

\begin{remark}
In the rest of these notes, we will only be interested in connected Real Riemann surfaces $(\Si,\tau)$, i.e\,\,Klein surfaces $\Si/\tau$ whose underlying topological surface is either non-orientable or has non-empty boundary (both properties may hold at the same time). We will no longer deal with dianalytic structures and therefore we shall often call the pair $(\Si,\tau)$ itself a Klein surface.
\end{remark}

\begin{multicols}{2}

\begin{exercise}\label{real_structures_on_proj_line}
Recall that $\Aut(\CP^1)\simeq \PGL(2;\C)$, acting by homographic transformations. By analysing what it means to be an element of order $2$ in that group, show that the only possible real structures on $\CP^1$ are $z\lmt\ov{z}$ and $z\lmt -\frac{1}{\ov{z}}$.
\end{exercise}

\begin{exercise}\label{real_curves_of_genus_one}
Let $a,b,c$ be three distinct complex numbers such that the set $\{a;b;c\}\subset\C$ is invariant under complex conjugation in $\C$. Consider the real elliptic curve defined by the equation $P(x,y)=0$ where $$P(x,y) = y^2-(x-a)(x-b)(x-c).$$ \textbf{a.} Show that if $a,b,c$ are real numbers, then the real plane curve defined by $P$ has two connected components (in the usual Hausdorff topology of $\R^2$).\\ \textbf{b.} Show that if for instance $a$ is not real, then the real plane curve defined by $P$ has only one connected component.\\ \textbf{c.} What is the topological type of the real projective curve defined by $P$ in each of the two cases above?
\end{exercise}

\begin{exercise}
Let $P\in\C[x,y]$ be a non-constant polynomial in two variables with complex coefficients. Show that $$X_P(\C) = \{(u,v)\in\C^2\ |\ P(u,v) =0\}$$ is a \textit{non-compact} Riemann surface.
\end{exercise}

\begin{exercise}\label{sheaf_theoretic_perspective}
Show that the sheaf $(p_*O_\Si)^\tau$ defined in \eqref{real_holomorphic_functions} is a sheaf of $\R$-algebras on $\Si/\tau$ and that, via \eqref{functions_on_Klein_surface}, it is a subsheaf of the sheaf of smooth functions on $\Si/\tau$.
\end{exercise}

\begin{exercise}
\textbf{a.} Show that if $(P,\si)$ is an orientable surface endowed with an orientation-reversing involution, then $P^\si$ is a union of at most $(g+1)$ circles. \textit{Indication}~: It suffices to prove that $P^\si$ is a $1$-dimensional manifold and to reproduce the proof of Theorem \ref{Harnack_thm} (the point being to prove Proposition \ref{topology_of_fixed_point_set} without using complex analytic structures and the normal form of holomorphic maps).\\ \textbf{b.} Show that, in contrast, an orientation-preserving involution of an orientable surface has $0$-dimensional fixed-point set.
\end{exercise}

\begin{exercise}
Prove the classical relation for connected sums of surfaces $$\chi(S_1\# S_2) = \chi(S_1) + \chi(S_2) -2.$$ Check the proof of Theorem \ref{Harnack_thm} again.
\end{exercise}

\begin{exercise}
Draw a picture of an orientation-reversing involution $\si$ of an oriented surface $P$ of genus $g$ having exactly $(g+1)$ circles of real points.
\end{exercise}

\begin{exercise}
Let $\Si$ be a compact connected Riemann surface of genus $g$.\\ \textbf{a.} Show that there are $\left[\frac{3g+4}{2}\right]$ topological types of real structures on $\Si$.\\ \textbf{b.} Compute this number for low values of $g$ and identify topologically in each case the surface $\Si/\tau$. \textit{Indication~:} Use Weichold's Theorem (Theorem \ref{Weichold_thm}) and distinguish cases according to whether $g$ is even or odd.
\end{exercise}

\begin{exercise}
\textbf{a.} Show that functions defined on open subsets of the closed upper half-plane $\C_+$ that satisfy assumptions (1), (2), (3) of the Schwarz reflection principle (Theorem \ref{Schwarz}) form a sheaf of $\R$-algebras.\\ \textbf{b.} Show that the locally ringed space $(A,D_A)$ thus defined is isomorphic to $(U/\si,(p_*O_U)^\si)$, where $U=A\cup \ov{A}$, $\si$ is complex conjugation, and $(p_*O_U)^\si$ is defined as in Exercise \ref{sheaf_theoretic_perspective}.
\end{exercise}

\begin{exercise}\label{example_of_morphism}
Let $A$ be an open subset of the closed upper half-plane $\C_+$. Show that a dianalytic function $f:A\lra \C_+$ is a homomorphism of Klein surfaces.
\end{exercise}

\begin{exercise}
A real curve $(\Si,\tau)$ such that $k=g+1$ is called a \textit{maximal curve} (or $M$-curve). Show that $k=g+1$ if and only if the sum of  Betti numbers of $\Si^\tau$ (do not forget to sum over the $(g+1)$ connected components of $\Si^\tau$) is equal to the sum of Betti numbers of $\Si$. Using $\mathrm{mod}\,2$ coefficients, this gives a way to generalize Harnack's inequality to higher-dimensional (smooth and projective) real algebraic varieties: the total $\mathrm{mod}\,2$ Betti number of the real part is lower or equal to the total $\mathrm{mod}\,2$ Betti number of the complex part; this is known as the Milnor-Thom inequality (\cite{Milnor,Thom}).
\end{exercise}

\end{multicols}

\section{The fundamental group of a real algebraic curve}\label{fund_gp_real_curves}

\subsection{A short reminder on the fundamental group of a Riemann surface}\label{reminder_cx_case}

The topological fundamental group $\piC$ of a connected pointed topological space $(X,x)$ is usually taken to be the group of homotopy classes of continuous maps $\gamma:[0;1] \lra X$ such that $\gamma(0) = \gamma(1)= x$ (also called a loop at $x$). In what follows, we will often simply denote by $\gamma$ the homotopy class of the loop $\gamma$ (this should result in no confusion). A key feature of this group is that \textit{it acts to the left on the fibre at $x$ of any topological covering map $p:Y\lra X$}. Indeed, this is a consequence of the following lemma (we assume throughout that $X$ is \textit{connected} and that $X$ and $Y$ are \textit{Hausdorff spaces}).

\begin{lemma}[Lifting of paths and homotopies]\label{lifting_of_paths_and_homotopies}
Let $\gamma:[0;1]\lra X$ be a path starting at $x$ (i.e.\ a continuous map $\gamma:[0;1]\lra X$ such that $\gamma(0)=x$) and let $p:Y\to X$ be a topological covering map.
\begin{enumerate}
\item Given $y\in p^{-1}(\{x\})$, there exists a unique path $\gat: [0;1] \lra Y$ such that $p\circ \gat = \ga$ and $\gat(0) = y$.
\item If $\ga'$ is homotopic to $\ga$ , then $\widetilde{\ga'}$ is homotopic to $\gat$. In particular, $\widetilde{\ga'}(1) = \gat(1)$.
\end{enumerate}
\end{lemma}

\noindent We refer for instance to \cite{Forster} for a proof of Lemma \ref{lifting_of_paths_and_homotopies}. The condition that $p\circ\gat=\ga$ implies that, given a point $y\in p^{-1}(\{x\})$ and a loop $\gamma$ at $x$, the element \begin{equation}\label{monodromy_action} \gamma\curvearrowright y := \gat(1)\end{equation} belongs to $p^{-1}(\{x\})$ and that $\gamma\curvearrowright y$ only depends on the homotopy class of $\ga$. The convention on path composition which we use in these notes is 
\begin{equation}\label{convention_on_paths}
(\ga \ast \ga')(t) = 
\left\{ 
\begin{array}{ccl}
\ga'(2t) & \mathrm{if} & 0\leq t \leq \frac{1}{2} \\
\ga(2t-1) & \mathrm{if} & \frac{1}{2} \leq t \leq 1
\end{array}
\right.
\end{equation}
\noindent (first travel $\ga'$, then $\ga$), so \eqref{monodromy_action} indeed defines a \textit{left} action of $\piC$ on the fibre at $x$ of any given topological cover of base $X$.
\begin{definition}[Monodromy action]\label{monodromy_action_def}
Given a topological cover $p:Y\lra X$ of base $X$, a point $x\in X$ and a point $y\in p^{-1}(\{x\})$, the left action $$\ga\curvearrowright y := \gat(1)$$ defined by lifting of loops at $x$ (i.e.\ $\gat$ is the unique path in $Y$ satisfying $p\circ \gat=\ga$ and $\gat(0)=y\in Y$) is called the \textbf{monodromy action}.
\end{definition}

\noindent The monodromy action is functorial in the following sense~: given $p:Y\lra X$ and $q:Z\to X$ two covers of $X$ and a homomorphism

$$
\xymatrix{
Y \ar[rr]^{\phi} \ar[rd]_p & & Z \ar[ld]^q \\
& X &
}
$$

\noindent between them, the induced map $$\phi_x: p^{-1}(\{x\}) \lra q^{-1}(\{x\})$$ is $\piC$-equivariant. Indeed, if $\gat$ satisfies $p\circ\gat=\ga$ and $\gat(0)=y\in p^{-1}(\{x\})$, then $\phi\circ\gat$ satisfies $$q\circ(\phi\circ\gat) = (q\circ\phi)\circ\gat = p \circ \gat = \ga$$ and $(\phi\circ\gat)(0) = \phi(y)$, so \begin{equation}\label{pi1_equivariance}\ga\curvearrowright \phi(y) = (\phi\circ\gat)(1) = \phi(\gat(1)) = \phi(\ga\curvearrowright y).\end{equation} This, in effect, turns $\piC$ into the \textit{automorphism group of the functor taking a cover of base $X$ to its fibre at $x$}. For a connected, Hausdorff and locally simply connected topological space $X$, however, the existence of a universal cover $\Xt$ simplifies matters much.

\begin{definition}[Universal cover]
Let $X$ be a connected topological space and let $p:\Xt\lra X$ be a topological covering map. It is called a \textbf{universal cover} for $X$ if $\Xt$ is connected and if, given $x\in X$ and $\xt\in \Xt$, the following universal property is satisfied~: for any \emph{connected} cover $q:Y\lra X$ and any choice of $y\in p^{-1}(\{x\})$, there exists a unique topological covering map $r:\Xt\lra Y$ such that $q\circ r = p$ and $q(\xt)=y$~:
$$
\xymatrix{
(\Xt,\xt) \ar[dd]_p \ar@{-->}[dr]^{\exists\, !\, r} & \\
& (Y,y) \ar[dl]^q \\
(X,x)
}
$$
\end{definition}

\noindent It is important to observe that the universal property above is one that is satisfied by the homomorphism of \textit{pointed} spaces  $p:(\Xt,\xt) \lra (X,x)$. We shall henceforth always assume that our universal covering maps are homomorphisms of pointed spaces, even if the notation does not explicitly reflect that a point $\xt\in p^{-1}(\{x\})$ has been chosen. We point out that the usual construction\footnote{Recall that part of the construction consists in endowing $\Xt$ with a topology making $p$ continuous, in which $\Xt$ is in fact connected -in particular, it can only cover other \textit{connected }coverings of $X$.} of $p:\Xt \lra X$ as the set of homotopy classes of paths $\eta: [0;1] \to X$ starting at $x$, 
$p$ being the map taking $\eta$ to $\eta(1)$, depends on the choice of $x\in X$ but then defines a distinguished point $\xt\in p^{-1}(\{x\})$, namely the homotopy class of the constant path at $x$. Moreover, there is a natural action of $\piC$ on $\Xt$ in this model, which is a \textit{right} action~: 
\begin{equation}\label{pi1_acts_on_universal_cover} 
\forall \eta\in\Xt,\ \forall \ga\in\piC,\ \eta \curvearrowleft \ga := \eta\ast \ga
\end{equation} 

\noindent (recall~: first travel $\ga$ then $\eta$, so $\eta\ast\ga$ is indeed a path starting at $x$ in $X$). Evidently, \eqref{pi1_acts_on_universal_cover} defines a right action of $\piC$ on $\Xt$ and it is an action by automorphisms of the covering map $p:\Xt\lra X$ since $$p(\eta\ast\ga) = (\eta\ast\ga)(1) = \eta(1) = p(\eta)$$ (recall that the group of automorphisms of a cover $p:\Xt\lra X$ is $$\Aut(\Xt/X) := \{ f:\Xt \lra \Xt\ |\ p\circ f = p\}$$ with group structure given by the composition of maps). Note that this construction defines a map from $\piC$ to $\Aut(\Xt/X)$. In Proposition \ref{usual_fund_gp}, we give a construction of that same map without using loops at $x$.

Because of its universal property, a \textit{pointed} universal cover $p:(\Xt,\xt)\lra (X,x)$, if it exists, is \textit{unique up to unique isomorphism}. It is a theorem that, if $p:\Xt \lra X$ is a topological covering map between Hausdorff spaces and if $\Xt$ is connected and simply connected, then $(\Xt,\xt)$ is a universal cover of $(X,x)$. Moreover, if $X$ is connected, Hausdorff and (semi-)locally simply connected (for instance, a connected topological manifold), the space $\Xt$ whose construction was sketched above is connected, Hausdorff and simply connected, so the original space $X$ in particular admits a universal cover (see for instance \cite{Forster}). We now drop this particular model of $\Xt$ and simply assume the existence of a universal cover $p:(\Xt,\xt)\lra (X,x)$, as well as knowledge of the fact that $\Xt$ is simply connected. This enables us to describe $\piC$ as a subgroup of $\Aut(\Xt)^{\op}$ (the group of self-homeomorphisms of $\Xt$ where the group structure is given by $f\cdot g = g\circ f$), which is simpler than entering the formalism of automorphism groups of functors. We recall that if $G$ is a group with composition law $\cdot$ then $G^{\op}$ is the group structure on the underlying set of $G$ obtained by decreeing that the product $g\cdot_{\op} h$ of $g$ and $h$ in $G^{op}$ is equal to $h\cdot g$. It is called the \textit{opposite group} of $G$.

\begin{proposition}\label{usual_fund_gp}
By the universal property of the universal cover, given $\gamma\in\piC$, there exists a unique $f:\Xt\lra \Xt$ covering $\Id_X$ on $X$ such that $f(\xt) = \ga\curvearrowright\xt$ (the monodromy action defined in \eqref{monodromy_action}). This in fact defines a group isomorphism
\begin{equation}\label{morphism_to_the_opposite_autom_gp}
\Psi:\begin{array}{ccc}
\piC & \overset{\simeq}{\lra} & \Aut(\Xt/X)^{\op}\\
\ga & \lmt & f
\end{array}.
\end{equation}
\end{proposition}

\begin{proof}
The map is well-defined and we first check that it is a group homomorphism, then we give an inverse. Denote $\Psi(\ga)=f$ and $\Psi(\ga')=g$. By definition, $\Psi(\ga\ast\ga')$ is the unique $h:\Xt\lra\Xt$ covering $\Id_X$ on $X$ such that $h(\xt)= (\ga\ast\ga')\curvearrowright \xt$. But $$(\ga\ast\ga')\curvearrowright\xt = \ga \curvearrowright (\ga'\curvearrowright\xt) = \ga \curvearrowright g(\xt)$$ so, as $g$ is $\piC$-equivariant for the monodromy action (see \eqref{pi1_equivariance}), $$ h(\xt) = \ga \curvearrowright g(\xt) = g(\ga \curvearrowright \xt) = g \big(f(\xt)\big) = (g\circ f)(\xt) = (f\cdot g)(\xt).$$ By uniqueness of such a map, $$\Psi(\ga\ast\ga') = f\cdot g = \Psi(\ga) \cdot \Psi(\ga')$$ so $\Psi:\piC\lra \Aut(\Xt/X)^{\op}$ is a group homomorphism. Next, we consider the map taking $f\in\Aut(\Xt/X)^{\op}$ to the homotopy class of $\ga:= p \circ\zeta$, where $\zeta$ is any path between $\xt$ and $f(\xt)\in p^{-1}(\{x\})$~: since $\Xt$ is simply connected, the homotopy class of $\zeta$ is uniquely determined and the homotopy class of $\ga$ is a well-defined element of $\piC$ that does not depend on the choice of $\zeta$. Moreover, $\zeta$ is a lifting of $\ga$ starting at $\xt$, so $\ga\curvearrowright \xt = \zeta(1)=f(\xt)$, which readily proves that this construction provides an inverse to $\Psi$.
\end{proof}

If one accepts the existence of a simply connected universal cover, the result of Proposition \ref{usual_fund_gp} may be used as a definition of $\piC$. The necessity to consider $\Aut(\Xt/X)^{\op}$ instead of $\Aut(\Xt/X)$ springs from the convention that we chose in \eqref{convention_on_paths} for composition of paths, which is in fact dictated by the theory of local systems (see \cite{Szamuely}, Remark 2.6.3). In particular, $\Aut(\Xt/X)^{\op}$ acts \textit{to the right} on $\Xt$. We now show that it acts \textit{to the left} on the fibre $q^{-1}(\{x\})$ of \textit{any} cover $q:Y\lra X$. This only uses the universal property of $p:(\Xt,\xt) \lra (X,x)$ and provides a description of the monodromy action \eqref{monodromy_action} making no use of loops at $x$.

\begin{proposition}\label{monodromy_action_bis}
Let $q:Y\lra X$ be a cover of $X$ and let $y$ be an element in $q^{-1}(\{x\})$. By the universal property of the universal cover $p:(\Xt,\xt) \lra (X,x)$, there exists a unique map $t_y:\Xt \lra Y$ 

$$
\xymatrix{
(\Xt,\xt) \ar[dr]^{t_y} \ar[dd]_p & \\
& (Y,y) \ar[dl]^q \\
(X,x) & 
}
$$

\noindent such that $q\circ t_y = p$ and $t_y(\xt)=y$. Given $f\in\Aut(\Xt/X)^{\op}$, let us set \begin{equation}\label{monodromy_action_without_loops} f\car y := (t_y\circ f)(\xt).\end{equation} This defines a left action of $\Aut(\Xt/X)^{\op}$ on $\fibreq$. Moreover, if $\phi:Y\lra Z$ is a homomorphism between two covers $q:Y\lra X$ and $r:Z\lra X$ of $X$, then the induced map $$\phi_x: \fibreq \lra r^{-1}(\{x\})$$ is $\Aut(\Xt/X)^{\op}$-equivariant.\\ Finally, $$f\car y = \ga \car y$$ where $\ga\in \piC$ satisfies $\Psi(\ga)=f$ under the isomorphism of Proposition \ref{usual_fund_gp} and $\ga\car y$ denotes the monodromy action defined in \eqref{monodromy_action}.
\end{proposition}

\begin{proof}
Let us first observe that $f\car y$ indeed belongs to $q^{-1}(\{x\})$. We compute~: 
\begin{equation}\label{action_on_fibre}
q(f\car y) = q\big((t_y\circ f)(\xt)\big) = (q\circ t_y) \circ f(\xt) = (p\circ f)(\xt) = p(\xt) = x.
\end{equation}

\noindent Moreover, one has $t_y\circ f = t_{f\car y}$. Indeed, $$q\circ (t_y\circ f) = (q\circ t_y) \circ f = p\circ f = p$$ and $(t_y\circ f)(\xt) = f\car y$ by definition of the latter, so $t_y\circ f = t_{f\car y}$ by uniqueness of the map $t_{f\car y}$. This in particular implies that 

$$
\begin{array}{ccccc}
(f\cdot g)\car y & = & (g\circ f) \car y & = & \big[t_y\circ (g\circ f)\big] (\xt) \\
& = & \big[(t_y\circ g) \circ f \big] (\xt) & = & \big[ t_{g\car y} \circ f \big] (\xt)\\ 
& = & f \car (g\car y) & &
\end{array}
$$ 

\noindent so we indeed have a left action of $\Aut(\Xt/X)^{\op}$ on $\fibreq$. The equivariance of $\phi_x:\fibreq \lra r^{-1}(\{x\})$ will follow from the commutativity of the diagram 

$$
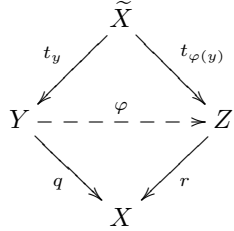
\xymatrix{
& \Xt \ar[dl]_{t_y} \ar[dr]^{t_{\phi(y)}} & \\
Y \ar@{-->}[rr]^{\phi}\ar[dr]_{q} & & Z \ar[dl]^r \\
& X& 
}
$$

\noindent We want to show that $\phi\circ t_y = t_{\phi(y)}$ (note that $r\circ t_{\phi(y)} = q \circ t_y$ because both are equal to $p:\Xt\lra X$ and that $r\circ \phi=q$ because $\phi$ is, by assumption, a homomorphism of covers). Indeed, $(\phi\circ t_y)(\xt)= \phi(t_y(\xt))= \phi(y) = t_{\phi(y)}(\xt)$ and $$r\circ (\phi\circ t_y) = (r\circ \phi)\circ t_y = q\circ t_y =p$$ so 
\begin{equation}\label{morphism_and_action}
\phi \circ t_y =t_{\phi(y)}
\end{equation}

\noindent by uniqueness of $t_{\phi(y)}$. As as consequence, $$\phi(f\car y) = \phi\big((t_y\circ f)(\xt)\big) = \big(\phi\circ t_y\big)\big(f(\xt)\big) = (t_{\phi(y)}\circ f)(\xt)=f\car \phi_(y)$$ so $\phi_x:=\phi|_{\fibreq}$ is indeed $\Aut(\Xt/X)^\op$-equivariant.\\ Then, by Proposition \ref{usual_fund_gp}, if $f=\Psi(\ga)$, one has $f(\xt) = \ga\car \xt$. Since $t_y$ is $\piC$-equivariant for the monodromy action (see \eqref{pi1_equivariance}), one therefore has $$f\car y = (t_y \circ f) (\xt) = t_y(\ga\car \xt) = \ga\car t_y(\xt) = \ga\car y.\qedhere$$
\end{proof}

\noindent A relation between the left action of $\Aut(\Xt/X)^\op$ on the fibre at $x$ of the universal cover $p:\Xt\lra X$ and the right action of $\Aut(\Xt/X)^\op$ on $\Xt$ is given in Exercise \ref{left_and_right_on_Xt}.

In what follows, we will continue to take for granted the existence of a simply connected universal cover for $X$ and mainly think of $\piC$ as the group $\Aut(\Xt/X)^{\op}$, the point being that it acts \textit{to the right} on $\Xt$ and \textit{to the left} on the fibre at $x$ of any cover of $X$. The notation $\piC$ is better for us than $\Aut(\Xt/X)^\op$ because it reflects the choice of a base point in $X$, which will be particularly important when considering real curves (see Propositions \ref{SDP_structure} and \ref{description_in_terms_of_paths}). From this point on, we will denote by $f$ an element of $\piC$ and write $y\cal f$ for the action of $f$ on an element $y\in\Xt$. With this in mind, we can now construct a cover of $X$ \textit{out of any left $\piC$-set $F$}, simply by setting $$Y:= \Xt\times_{\piC} F := \piC \bs (\Xt \times F),$$ the quotient space of  the left $\piC$-action on $\Xt\times F$ defined, for any $f\in\piC$ and any $(y,v) \in \Xt\times F$, by $$f \car (y,v) = (y\cal f^{-1}, f\car v).$$ This construction provides a quasi-inverse to the fibre functor (of which $\piC$ is the automorphism group) and this yields the following Galois-theoretic approach to topological covering spaces, first stated this way by Grothendieck (\cite{SGA1}).

\begin{theorem}[Galois theory of topological covering spaces]\label{Galois_theory_of_covers}
Let $X$ be a connected, Hausdorff and semi-locally simply connected topological space. The category $\Cov_X$ of Hausdorff topological covering spaces of $X$ is equivalent to the category of discrete left $\piC$-sets. An explicit pair of quasi-inverse functors is provided by the fibre-at-$x$ functor 

$$\Fib_x: \begin{array}{ccc}
\Cov_X & \lra & \piC-\Sets \\
(q: Y\lra X) & \lmt & q^{-1}(\{x\}) 
\end{array}
$$

\noindent and the functor

$$Q: \begin{array}{ccc}
 \piC-\Sets & \lra & \Cov_X  \\
F & \lmt & \Xt \times_{\piC} F
\end{array}\, .
$$

\noindent The map $$\begin{array}{ccc}
\{\mathrm{subgroups\ of}\ \piC\} & \lra & \{\mathrm{transitive}\ \piC-\Sets\} \\
H & \lmt & F:=\piC/H
\end{array}$$ sets up a bijection between subgroups of $\piC$ and connected covers of $X$. Under this bijection, Galois covers of $X$ (i.e\,\,connected covers $q:Y\lra X$ for which the action of $\Aut(Y/X)$ on any given fibre of $q$ is transitive) correspond to normal subgroups of $\piC$.
\end{theorem}

\noindent In particular, a cover $q:Y\lra X$ is connected if and only if $q^{-1}(\{x\})$ is a transitive $\piC$-set and it is Galois if and only if $q^{-1}(\{x\})$ is a transitive $\piC$-set which admits a compatible group structure (i.e.\ a group structure which turns this set into a quotient group of $\piC$). We refer for instance to \cite{Szamuely} for a proof of Theorem \ref{Galois_theory_of_covers}. Note that the important part is knowing that the functor $\Fib_x$ establishes an equivalence of categories between covers of $X$ and left $\piC$-sets and that the universal cover is only being used, really, to produce a quasi-inverse to $\Fib_x$. Indeed, this is exactly what generalizes to covers of real curves, as we shall see in Subsection \ref{fund_gp_Klein_surface}. Another interesting point is that if one considers the restriction of $\Fib_x$ to the full subcategory of \textit{connected} Hausdorff covers of $X$, then by the universal property of $\Xt$, we have $\Fib_x(Y) = \Hom_{\Cov_X}(\Xt;Y)$. In other words, the functor $\Fib_x$ is representable by the universal cover $(\Xt\,\xt)$ (see Exercise \ref{represent_the_fibre_functor}).

Finally, if $X$ is a Riemann surface and $p:Y\lra X$ is a topological covering map with $Y$ a Hausdorff space, there exists a unique holomorphic structure on $Y$ turning $p$ into a holomorphic map (see for instance \cite{Forster} or \cite{Szamuely}). So \textit{the category of Hausdorff topological covers of a connected Riemann surface $X$ is equivalent to the category of holomorphic covering spaces of $X$}. In particular, $X$ admits a universal cover \textit{in the category of Riemann surfaces} and the topological fundamental group $\piC=\Aut(\Xt/X)^\op$ acts on $\Xt$ by holomorphic transformations, \textit{classifying holomorphic covers of $X$ in the sense of Theorem \ref{Galois_theory_of_covers}}. We shall see later that this point of view has an exact analogue for Klein surfaces (Subsection \ref{Galois_over_R}).

\subsection{The fundamental group of a Klein surface}\label{fund_gp_Klein_surface}

By Theorem \ref{Galois_theory_of_covers}, the fundamental group $\piC=\Aut(\Xt/X)^{\op}$ of a Riemann surface behaves like a Galois group for the category of complex analytic covers of $X$. The fundamental group of a Klein surface $(X,\si)$, whose definition we give in the present subsection (Definition \ref{def_fund_gp_real_curve}), plays the same role \textit{vis-\`a-vis} the category of real covering maps of base $(X,\si)$. In what follows, we denote by $\Aut^\pm(X)$ the group of holomorphic and anti-holomorphic transformations of a given Riemann surface $X$ (this is indeed a group; if $X$ has a real structure $\si$, it is isomorphic to the semi-direct product $\Aut(X)\rtimes_{\si}\Z/2\Z$, where $\Z/2\Z=\{\Id_X;\si\}$ acts on the group $\Aut(X)$ of holomorphic transformations of $X$ by $\si\car f = \si\circ f \circ \si^{-1}$, see Exercise \ref{hol_and_anti_hol_bijections}). The fundamental group of $(X,\si)$ will be a subgroup of $(\Aut^\pm(\Xt))^\op$ but will not always be a semi-direct product and, at any rate, not canonically (for instance, the semi-direct product structure defined in Proposition \ref{SDP_structure} depends on the choice of the base point\footnote{In fact, it only depends on the connected component of $X^\si$ containing that real point, see Exercise \ref{isomorphic_SDP_structures}.}). We begin by studying the category of covers of $(X,\si)$ for which we want to set up a Galois theory.

\begin{definition}[Real covers and their homomorphisms]
A \textbf{real cover of a Klein surface $(X,\si)$} is a homomorphism of Klein surfaces $q:(Y,\tau) \lra (X,\si)$ (i.e.\ a holomorphic map between $Y$ and $X$ satisfying $q\circ\tau = \si \circ q$) which is also a topological covering map.\\ A \textbf{homomorphism of real covers of $X$} is a homomorphism of Klein surfaces $\phi:(Y,\tau) \lra (Y',\tau')$ which is also a homomorphism of covers of $X$ (i.e.\ a holomorphic map satisfying $q'\circ\phi = q$, where $q:(Y,\tau)\lra (X,\si)$ and $q':(Y',\tau') \lra (X,\si)$).
\end{definition}

\noindent The condition that $\phi:(Y,\tau)\lra (Y,\tau')$ be a homomorphism of real covers of $(X,\si)$ is equivalent to the commutativity of the following diagram

$$
\xymatrix{
Y\ar[rrr]_{\tau} \ar[dr]^{\phi} \ar[ddr]_q & & & Y \ar[ld]_{\phi} \ar[ldd]^q  \\
& Y' \ar[d]^{q'} \ar[r]_{\tau'} & Y' \ar[d]_{q'} & \\
& X \ar[r]_\si & X & 
}
$$

\noindent or, more compactly, 

$$
\xymatrix{
(Y,\tau) \ar[rr]^{\phi}\ar[dr]_{q} & & (Y',\tau') \ar[dl]^{q'} \\
& (X,\si). & 
}
$$

A natural question to ask is whether there exists a group $\pi$ such that the category of real covers of $(X,\si)$ is equivalent to the category of discrete left $\pi$-sets, on the model of Theorem \ref{Galois_theory_of_covers}. Given a real cover $q:(Y,\tau) \lra (X,\si)$, we already know, by Proposition \ref{monodromy_action_bis}, that $\piC$ acts on $\fibreq$~: if $y$ is any element in $\fibreq$, then there exists a unique map $$t_y:(\Xt,\xt) \lra (Y,y)$$ such that $q\circ t_y = p$, where $p:(\Xt,\xt)\lra (X,x)$ is a fixed universal cover of $X$. Now, since $q:(Y,\tau)\lra (X,\si)$ is a homomorphism of Klein surfaces, one has a bijection of discrete sets $$\tau|_{\fibreq}: \fibreq \overset{\simeq}{\lra} q^{-1}(\{\si(x)\})$$ \textit{and this bijection can be used to define an action on $\fibreq$ of a group strictly larger than $\piC$}. Indeed, assume that $f:\Xt\lra \Xt$ is a self-diffeomorphism of $\Xt$ which covers not $\Id_X$ but the real structure $\si$ in the following sense~:

$$
\begin{CD}
\Xt @>f>> \Xt \\
@VVpV @VVpV \\
X @>\si>> X.
\end{CD}
$$

\noindent Then $$q\big((t_y\circ f)(\xt)\big) = \big((q\circ t_y)\circ f\big)\big(\xt\big)=(p\circ f)(\xt) = (\si\circ p)(\xt) = \si(x)$$ so $(t_y\circ f)(\xt) \in q^{-1}(\{\si(x)\})$ and therefore 
\begin{equation}\label{action_on_the_fibre_in_real_case}
f\car y := (\tau^{-1} \circ t_y \circ f) (\xt)\in q^{-1}(\{x\}).
\end{equation}

\noindent In other words, the presence of a real structure $\tau$ on $Y$, compatible with the projection to $(X,\si)$, implies that the group
\begin{equation}\label{real_pi1}
\piR := \{ f:\Xt\lra\Xt\ |\ \exists\,\si_f\in\{\Id_X;\si\},\, p\circ f = f\circ \si_f\}^\op \subset \Aut^\pm(\Xt)^\op
\end{equation}

\noindent acts to the left on the fibre of $q$ at $x$. It is easy to see, but we will prove below, that this group indeed contains $\piC$ as the subgroup of elements $f$ for which $\si_f=\Id_X$ (in particular, this is a normal subgroup since, if $g$ covers $\mu\in\{\Id_X;\si\}$, then $g^{-1}\circ f \circ g$ covers $\mu^{-1} \circ \Id_X \circ \mu=\Id_X$). An alternate definition of $\piR$ is 
\begin{equation}\label{real_pi1bis}
\piR \quad = \bigsqcup_{\mu  \in\{\Id_X;\si\}} \{f:\Xt \lra \Xt\ |\ p\circ f = \mu \circ p\}.
\end{equation}

\noindent A few things need to be checked here and we formulate two separate results.

\begin{proposition}\label{piR_is_a_group}
The set $\piR$ defined in \eqref{real_pi1} or \eqref{real_pi1bis} is a subgroup of $\Aut^\pm(\Xt)^\op$~: $f\in\piR$ acts holomorphically on $\Xt$ if $\si_f=\Id_X$ and anti-holomorphically if $\si_f=\si$. There is a well-defined group homomorphism 
\begin{equation}\label{augmentation}
\alpha:
\begin{array}{ccc}
\piR & \lra & \{\Id_X;\si\}^\op \simeq \Z/2\Z \\
f & \lmt & \si_f
\end{array}
\end{equation}
\noindent whose kernel is $$\ker\alpha = \{f:\Xt\lra \Xt\ |\ p\circ f = p\}^\op = \piC.$$ In particular, $\piC$ is a normal subgroup of $\piR$.
\end{proposition}

\noindent We will see later that $\alpha$ is surjective, so that $\piC$ is of index $2$ in $\piR$ (see Proposition \ref{extension}).

\begin{proposition}\label{piR_acts_on_fibres}
Let $(X,\si)$ be a Klein surface and let $q:(Y,\tau)\lra (X,\si)$ be a real cover. Recall that, given $y\in\fibreq$, there exists a unique map $t_y:\Xt \lra Y$ such that $t_y(\xt) = y$ and $q\circ t_y = p$, where $p:(\Xt,\xt)\lra(X,x)$ is a fixed universal cover of $X$. The group $\piR$ defined in \eqref{real_pi1} acts to the left on the fibre $\fibreq$ by $$f\car y := (\tau_f^{-1} \circ t_y \circ f) (\xt)$$ where $\tau_f=\Id_Y$ if $\si_f=\Id_X$ and $\tau_f=\tau$ if $\si_f=\si$. Moreover, if $\phi:(Y,\tau) \lra (Y',\tau')$ is a homomorphism between two real covers $q:(Y,\tau) \lra (X,\si)$ and $q':(Y',\tau') \lra (X,\si)$ of $(X,\si)$, then the induced map $$\phi_x: q^{-1}(\{x\}) \lra (q')^{-1}(\{x\})$$ is $\piR$-equivariant.
\end{proposition}

\noindent Proposition \ref{piR_acts_on_fibres} means, in effect, that $\piR$ is the automorphism group of the fibre-at-$x$ functor defined on the category of real covers of $(X,\si)$. We shall not, however, make use of this formalism.

\begin{proof}[Proof of Proposition \ref{piR_is_a_group}]
We recall that the group structure on $\Aut^\pm(\Xt)^\op$ is given by $f\cdot g= g\circ f$. First, note that, given $f\in\Aut^\pm(\Xt)^\op$, the element $\mu\in\{\Id_X;\si\}$ such that $p\circ f = \mu\circ p$, if it exists, is unique. In particular,  the map $\alpha$ in \eqref{augmentation} is well-defined and, since $p$ is a holomorphic covering map, $f$ is holomorphic if $\si_f=\Id_X$ and anti-holomorphic if $\si_f=\si$. Then, if $f,g\in\piR$, one has 

$$
\begin{array}{rclclclcl}
p\circ (f\cdot g) & = & p\circ (g\circ f) & = & (p\circ g) \circ f & = & (\si_g\circ p)\circ f\\
& = & \si_g \circ (p\circ f) & = & \si_g \circ (\si_f\circ p) & = & (\si_g\circ \si_f)\circ p\\
& = & (\si_f \cdot \si_g) \circ p
\end{array}
$$

\noindent so $(f\cdot g)\in\piR$ and, by uniqueness, $\si_{f\cdot g} = \si_f \cdot \si_g$. Likewise, if $f\in\piR$, then on the one hand $$p\circ (f\circ f^{-1}) = p\circ \Id_{\Xt} = \Id_X\circ p=p$$ and on the other hand $$(p\circ f)\circ f^{-1} = (\si_f\circ p) \circ f^{-1}$$ so $p\circ f^{-1} = \si_f^{-1}\circ p$, which proves that $f^{-1}\in\piR$. Therefore, $\piR$ is a subgroup of $\Aut^\pm(\Xt)^\op$ and $\alpha$ is a group homomorphism. Since $$\piC=\{f:\Xt\lra\Xt\ |\ p\circ f = \Id_X\circ p\}^\op,$$ $\piC$ is the kernel of $\alpha$.
\end{proof}

\begin{proof}[Proof of Proposition \ref{piR_acts_on_fibres}]
The proof that follows should be compared to that of Proposition \ref{monodromy_action_bis}. We have already observed in \eqref{action_on_fibre} and \eqref{action_on_the_fibre_in_real_case} that $$f\car y:=(\tau_f^{-1}\circ t_y \circ f)(\xt) \in q^{-1}(\{x\})$$ for all $f\in\piR$. Moreover, one has $t_{f\car y} = \tau_f^{-1}\circ t_y \circ f$. Indeed, on the one hand, $(\tau_f^{-1}\circ t_y \circ f)(\xt) = f\car y$ by definition of the latter and, on the other hand,

$$
\begin{array}{rcl}
q\circ (\tau_f^{-1} \circ t_y \circ f) & = & (q\circ \tau_f^{-1}) \circ (t_y\circ f) \\
& = & (\si_{f}^{-1} \circ q) \circ (t_y \circ f) \\
& = & \si_f^{-1} \circ (q\circ t_y) \circ f \\
& = & \si_f^{-1} \circ p\circ f\\
& = & \si_f^{-1} \circ \si_f \circ p \\
& = & p
\end{array}
$$

\noindent so $\tau_f\circ t_y \circ f = t_{f\car y}$ by uniqueness of the map $t_{f\car y}$. This in particular implies that $\piR$ acts to the left on $\fibreq$. Indeed, having noted that $\si_{f\cdot g}= \si_f \cdot \si_g$ implies that $\tau_{f\cdot g} = \tau_f \cdot \tau_g$ by definition of $\tau_f$, one has

$$
\begin{array}{rclcl}
(f\cdot g) \car y & = & \big[\tau_{f\cdot g}^{-1} \circ t_y \circ (f\cdot g)\big](\xt) & = &\big[ (\tau_f \cdot \tau_g)^{-1} \circ t_y \circ (f\cdot g) \big](\xt)\\
& = & \big[(\tau_g\circ \tau_f)^{-1} \circ t_y \circ (g\circ f) \big](\xt)& = & \big[ \tau_f^{-1} \circ (\tau_g^{-1} \circ t_y \circ g) \circ f \big](\xt)\\
& = & \big[\tau_f^{-1} \circ \tau_{g\car y} \circ f \big](\xt)\\
& = & f \car (g\car y).
\end{array}
$$

\noindent Likewise, the equivariance of $\phi_x:\fibreq \lra (q')^{-1}(\{x\})$ will follow from the commutativity of the diagram 

$$
\xymatrix{
& \Xt \ar[dl]_{t_y} \ar[dr]^{t_{\phi(y)}} & \\
(Y,\tau) \ar[rr]^{\phi}\ar[dr]_{q} & & (Y',\tau') \ar[dl]^{q'} \\
& (X,\si)& 
}
$$

\noindent Indeed we showed in \eqref{morphism_and_action} that $\phi\circ t_y = t_{\phi(y)}$. As as consequence, for all $f\in\piR$,

$$
\begin{array}{rcl}
\phi(f\car y) & = & \phi\big((\tau_f^{-1}\circ t_y\circ f)(\xt)\big) \\
& = & \big[(\phi\circ \tau_f^{-1}) \circ t_y \circ f\big](\xt) \\
& = & \big[ \big((\tau'_f)^{-1} \circ \phi\big) \circ t_y \circ f\big] (\xt) \\
& = & \big[ (\tau'_f)^{-1} \circ (\phi\circ t_y) \circ f\big] (\xt) \\
& = & \big[ (\tau'_f)^{-1} \circ t_{\phi(y)}\circ f \big] (\xt)\\
& =& f\car \phi(y)
\end{array}
$$ 

\noindent so $\phi_x:=\phi|_{\fibreq}$ is indeed $\piR$-equivariant.
\end{proof}

\noindent We can now give the group $\piR$ a name. Perhaps an alternate notation for $\piR$ would be $(\Aut_\si(\Xt/X))^\op$ but, just like for the fundamental group of a complex curve, it fails to take into account the base point and this would be problematic later on (see Propositions \ref{SDP_structure} and \ref{description_in_terms_of_paths}). One last comment before giving the definition of the central notion of the present notes is this~: there are various more conceptual ways to think about the fundamental group of a Klein surface (e.g.\ by further developping the categorical point of view and talking about automorphisms of fibre functors, or by considering orbifolds and using the general theory of orbifold fundamental groups), here we adopt a rather hands-on approach, thinking only about real covers of Klein surfaces and the topological universal covering of the underlying Riemann surface. We hope that this will serve as a gentle introduction to the subject and will make the reader eager to know more about this rich topic.

\begin{definition}[Fundamental group of a real curve]\label{def_fund_gp_real_curve}
Let $(X,\si)$ be a Klein surface and let $x\in X$. We fix a universal cover $p:(\Xt,\xt)\lra (X,x)$. The \textbf{fundamental group} of the Klein surface $(X,\si)$ is the group $$\piR := \{ f:\Xt\lra\Xt\ |\ \exists\,\si_f\in\{\Id_X;\si\},\, p\circ f = f\circ \si_f\}^\op$$ with group structure given by $(f\cdot g) := g\circ f$.
\end{definition}

\noindent As briefly mentioned earlier, a more conceptual definition would be to simply take $\piR$ to be the automorphism group of the fibre-at-$x$ functor on the category of real covers of $(X,\si)$ but, because of the existence of a universal cover $p:(X,\xt)\lra (X,x)$, this would eventually coincide with the definition just given. It should be noted that the universal cover $\Xt$ of $X$ is \textit{not} a real cover in general~: $X$ admits a connected universal cover in the category of Riemann surfaces but $(X,\si)$ does not, in general, admit a \textit{connected} universal cover in the category of Klein surfaces (one way to see this is to use the uniformisation theorem for Riemann surfaces and show, for instance, that a real structure on the Poincar\'e upper half-plane is necessarily conjugate to $z\lra -\ov{z}$ by an element of $\PSL(2;\R)$, so in particular it has a connected set of real points, whereas there are Klein surfaces of genus $g\geq 2$ which admit real structures with no real points, hence cannot be covered by the upper half-plane in the category of Klein surfaces). This also means that there is in general no hope of having $\piR$ act on $\Xt$ by automorphisms of real covers. However, the group $\piR$ acts on $\Xt$ by transformations which are either holomorphic or anti-holomorphic; in other words, $\piR$ is a subgroup of $(\Aut^\pm(\Xt))^\op$.

The next lemma will be useful several times in the remainder of the present subsection. 

\begin{lemma}\label{lift_of_real_structure}
Let $(X,\si)$ be a Klein surface and let $p:(\Xt,\xt)\lra (X,x)$ be a universal cover of $X$, where $x$ is any point in $X$.
\begin{enumerate}
\item Given any $y\in \fibrepsi$,there exists a unique anti-holomorphic map $\sit:\Xt \lra \Xt$ such that $\sit(\xt)=y$ and the diagram

$$
\begin{CD}
\Xt @>\sit>> \Xt \\
@VVpV @VVpV \\
X @>\si>> X
\end{CD}
$$

\noindent is commutative.
\item If $x\in\Fix(\si)$ and $$y:=\xt\in\fibre=\fibrepsi,$$ then the anti-holomorphic map $\sit$ defined in \emph{(1)} satisfies $\sit^2=\Id_{\Xt}$~: in other words, it is a real structure on $\Xt$.
\end{enumerate}
\end{lemma}

\begin{proof}
The lemma says that $\si$ always lifts to an anti-holomorphic transformation of $\Xt$ covering $\si$ but it does not lift to an \textit{involutive} such transformation in general.
\begin{enumerate}
\item Consider the (anti-holomorphic) covering map $$\si\circ p: \big(\Xt, \xt\big) \lra \big(X,\si(x)\big).$$ Since $\Xt$ is simply connected, $\si\circ p$ is a topological universal cover of $X$, hence the existence, given $y\in\fibrepsi$, of a unique covering map $$\sit: (\Xt,\xt) \lra (\Xt,y)$$ satisfying $p\circ \sit = \si\circ p$. Since $(\Xt,y)$ is simply connected, the covering map $\sit:(\Xt,\xt)\lra (\Xt,y)$ must be one-to-one, so it is a diffeomorphism. Moreover, since $\si$ is anti-holomorphic and $p$ is a holomorphic covering map, the relation $p\circ \sit=\si\circ p$ implies that $\sit$ is anti-holomorphic.
\item If $x\in\Fix(\si)$ and $y=\xt$, one has that $\sit^2:\Xt\lra \Xt$ covers $\si^2=\Id_X$ and takes $\xt$ to $\xt$. By uniqueness of such a map, one has $\sit^2=\Id_{\Xt}$.
\end{enumerate}
Note that point (2) above shows that the choice of a real point in $X$, if possible, defines a real structure on $\Xt$ which turns the universal covering map $p:\Xt\lra X$ into a homomorphism of Klein surfaces. A different real point in $X$ may lead to a non-isomorphic real structure on $\Xt$ (see also the remark after Proposition \ref{SDP_structure}).
\end{proof}

\noindent The first application of Lemma \ref{lift_of_real_structure} is a property of the fundamental group of a real curve which, albeit immediate from our definition of this group, is nonetheless key to all that follows, in particular the notion of \textit{representation} of the fundamental group of a real curve to be presented in Section \ref{section_reps_fund_gp}.

\begin{proposition}\label{extension}
Let $\alpha:\piR\lra \{\Id_X;\si\}^\op \simeq  \Z/2\Z$ be the group homomorphism defined in Proposition \ref{piR_is_a_group}. There is a short exact sequence, called the \textbf{homotopy exact sequence},
\begin{equation}\label{homotopy_exact_sequence}
1\lra \piC \lra \piR \overset{\alpha}{\lra} \Z/2\Z \lra 1
\end{equation}

\noindent showing in particular that $\piC$ is an index two subgroup of $\piR$.
\end{proposition}

\begin{proof}
Based on Proposition \ref{piR_is_a_group}, it only remains to prove that $\alpha$ is surjective. This is a consequence of point (1) in Lemma \ref{lift_of_real_structure}.
\end{proof}

The next application of Lemma \ref{lift_of_real_structure} is a proof of the fact that, when the base point $x$ is a fixed point of $\si$, the group $\piR$ has a canonical semi-direct product structure (see also Exercises \ref{more_on_SDP_structures} and \ref{isomorphic_SDP_structures}).

\begin{proposition}\label{SDP_structure}
Let $x\in\Fix(\si)$ be a real point of $X$. By Lemma \ref{lift_of_real_structure}, there exists a real structure $\sit$ on $\Xt$ sending $\xt$ to $\xt$. This $\sit$ defines a left action of $\{\Id_{\Xt};\sit\}^\op \simeq \Z/2\Z$ on $\piC=\Aut(\Xt/X)^\op$ by conjugation \begin{equation}\label{involution_on_complex_pi_one} \sit\cdot f = \sit f \sit^{-1}\end{equation} and $\piR$ is isomorphic to the associated semi-direct product $$\piC\rtimes_{\sit} \Z/2\Z.$$
\end{proposition}

\noindent Note that it is implicit in the statement above that the symbol $\cdot$ designates the group action of $\Z/2\Z$ on $\piC$ and that $\sit f \sit^{-1}$ denotes the product of $\sit$, $f$ and $\sit^{-1}$ in $\piR$. From now on, we will use this notation instead of the one used until now, which would have been $\sit\car f = \sit \cdot f \cdot \sit^{-1}=\sit^{-1}\circ f\circ\sit$.

\begin{proof}[Proof of Proposition \ref{SDP_structure}]
Observe first that $\sit f\sit^{-1}$ covers $\Id_X$ if so does $f$ and is holomorphic if so is $f$. This defines a left action of $\Z/2\Z$ on $\piC$ because, as is easily checked, $\widetilde{\eps\eta} = \widetilde{\eps}\, \widetilde{\eta}$ for all $\eps,\eta\in\{\Id_X;\si\}$. If $g\in\piR$ covers $\si$, then $g = (g\sit^{-1})\sit$, with $g\sit^{-1}$ covering $\Id_X$, so the map 
\begin{equation}\label{SDP_isom}
\begin{array}{ccc}
\piC \rtimes_{\sit} \{\Id_{\Xt};\sit\} & \lra & \piR \\
(f,\eps) & \lmt & f\eps
\end{array}
\end{equation}

\noindent is surjective. Since, for $f\in\piC$, $f\eps$ covers $\si$ on $X$ if and only if $\eps=\sit$, it is also injective. Finally, the product $$(f,\eps)(g,\eta) = \big(f(\eps\cdot g) , \eps\eta\big) = (f\eps g\eps^{-1}, \eps\eta)$$ in $\piC\rtimes_{\sit} \{\Id_{\Xt};\sit\}$ is sent, under \eqref{SDP_isom}, to $$(f\eps g\eps^{-1})(\eps\eta) = (f\eps)(g\eta)$$ so \eqref{SDP_isom} is indeed a group isomorphism.
\end{proof}

\noindent We recall that a semi-direct product structure on $\piR$ is the same as a splitting $\Z/2\Z\lra \piR$ of the homotopy exact sequence \eqref{homotopy_exact_sequence} and this is the same as an anti-holomorphic element of order $2$ in $\piR$, i.e.\ \textit{a real structure on $\Xt$, covering $\si$ on $X$}. Thus, Lemma \ref{lift_of_real_structure} and Proposition \ref{SDP_structure} show that \textit{a real point of $(X,\si)$} gives rise to a splitting of the homotopy exact sequence $$1\lra \piC \lra \piR \overset{\alpha}{\lra} \Z/2\Z \lra 1.$$ Of course, different real points may lead to different semi-direct product structures but understanding how exactly it is so is not quite an elementary story and we refer to \cite{Mochizuki} for results on this.

The final application of Lemma \ref{lift_of_real_structure} is a description of the fundamental group of $(X,\si)$ using loops at $x$ (and, when $\si(x)\neq x$, paths from $x$ to $\si(x)$). We will see that this description is different depending on whether $x$ is a real point of $(X,\si)$. This is the reason why we have given a definition of $\piR$ without any reference to paths in $X$, for it provides a more intrinsic description of the group $\piR$, in which it is not necessary to distinguish cases depending on whether $x$ is a real point or not. When $x\in \Fix(\si)$, the description of $\piR$ in terms of loops at $x$ is in fact a corollary of Proposition \ref{SDP_structure}.

\begin{corollary}\label{SDP_structure_in_terms_of_paths}
Think of $\piC$ as the group of homotopy classes of loops at $x\in\Fix(\si)$. Then the left action of $\{\Id_X;\si\}^\op\simeq \Z/2\Z$ on $\piC$ defined in \eqref{involution_on_complex_pi_one} is given, for any $\ga\in\piC$, by $$\si_*\ga = \si^{-1} \circ \ga$$ and then $$\piR \simeq \piC \rtimes_{\si} \Z/2\Z$$ with respect to this action.
\end{corollary}

\begin{proof}
By Proposition \ref{SDP_structure}, it suffices to show that $\Psi(\si_*\ga) = \sit f \sit^{-1}$, where $\Psi:\piC \lra \Aut(\Xt/X)^\op$ is the map defined in Proposition \ref{usual_fund_gp}. By definition of this map, $\Psi(\si_*\ga)$ is the unique automorphism of universal cover $p:\Xt\lra X$ sending $\xt$ to $(\si_*\ga)\car \xt$. But $(\si_*\ga)\car \xt = (\widetilde{\si^{-1}\circ\ga})(1)$, where $\widetilde{\si^{-1}\circ\ga}$ is the unique path in $\Xt$ such that $p\circ(\widetilde{\si^{-1}\circ\ga})=\si^{-1}\circ\ga$ and $(\widetilde{\si^{-1}\circ\ga})(0) = \xt$. Let us denote, as usual, $\gat$ the unique path in $\Xt$ such that $p\circ\gat=\ga$ and $\gat(0)=\xt$. Since, by definition of $\sit$ (see Lemma \ref{lift_of_real_structure}), $(\sit^{-1}\circ\ga)(0)=\sit^{-1}(\gat(0)) = \sit^{-1}(\xt)=\xt$ and $p\circ(\sit^{-1}\circ\gat)=(p\circ\sit^{-1})\circ\gat=(\si^{-1}\circ p)\circ\gat=\si^{-1}\circ(p\circ\gat)=\si^{-1}\circ\ga$, one has $\widetilde{\si^{-1}\circ\ga} = \sit^{-1} \circ \gat$, by uniqueness of such a lifting. As a consequence, $(\si_*\ga)\car\xt = (\sit^{-1}\circ\gat)(1) = \sit^{-1}(\gat(1)) = \sit^{-1}(\ga\car\xt) = \sit^{-1} (f(\xt))$, where $f=\Psi(\ga)$. Since $\sit(\xt)=\xt$, one has $$(\si_*\ga)\car\xt = (\sit^{-1}\circ f)(\xt)=(\sit^{-1}\circ f\circ \sit)(\xt).$$ Moreover, $\sit^{-1}\circ f\circ \sit$ is an automorphism of the cover $\Xt/X$, so $\Psi(\si_*\ga) = \sit^{-1}\circ f\circ \sit$ by definition of the map $\Psi$. Then we observe that $\Psi(\si_*\ga)$ is actually equal to the product $\sit f\sit^{-1}$ in $\piR\subset\Aut^\pm(\Xt)^\op$.
\end{proof}

\begin{remark}\label{SDP_structure_is_canonical_for_x}
Observe that if we take $\Xt$ to be the usual set of homotopy classes of paths starting at $x\in X^\si$, then it is natural to take $\xt$ to be the homotopy class of the constant path at $x$. In particular, $\sit$ is entirely determined by the choice of $x$, so the semi-direct products of Proposition \ref{SDP_structure} and Corollary \ref{SDP_structure_in_terms_of_paths} are canonically isomorphic.
\end{remark}

When $x\notin\Fix(\si)$, there is a different description of $\piR$ in terms of paths in $X$, given in Proposition \ref{description_in_terms_of_paths}. It requires the next lemma, which is interesting in itself.

\begin{lemma}\label{lift_of_real_paths_to_autom}
Let $x$ be a point of $X$ such that $\si(x)\neq x$ and let $\eta:[0;1]\lra X$ be a path from $x$ to $\si(x)$. By Lemma \ref{lifting_of_paths_and_homotopies}, there exists a unique path $\etat:[0;1]\lra \Xt$ such that $p\circ\etat=\eta$ and $\etat(0) = \xt$. Then, there exists a unique anti-holomorphic map $\sit:\Xt\lra \Xt$ such that $p\circ \sit = \si\circ p$ and $\sit(\xt)= \etat(1)$.
\end{lemma}

\begin{proof}
This is a direct application of part (1) of Lemma \ref{lift_of_real_structure}. Note that $\sit^2\neq\Id_{\Xt}$ in general (see also Exercise \ref{more_on_SDP_structures} for related considerations).
\end{proof}

\noindent In the next proposition and its proof, we shall denote by $\ga\car \xt$ the element $\gat(1)$, where $\ga$ is a path in $X$ starting at $x$ and going either to $x$ or to $\si(x)$, and $\gat$ is the unique lifting of $\ga$ starting at $\xt$.

\begin{proposition}\label{description_in_terms_of_paths}
Let $x$ be a point of $X$ such that $\si(x)\neq x$. Let $\Ga$ be the set of homotopy classes (with fixed endpoints) of paths of the form $$\{\zeta:[0;1]\lra X\ |\ \zeta(0)=\zeta(1)=x\} \sqcup \{\eta:[0;1]\lra X\ |\ \eta(0)=x\ \mathrm{and}\ \eta(1)=\si(x)\}$$ and set, for all $\ga,\ga'$ in $\Ga$, $$\ga \boxtimes \ga' = \left\{ \begin{array}{cl}
\ga \ast \ga' & \mathrm{if}\ \ga'(1)=x \\
(\si\circ\ga) \ast \ga' & \mathrm{if}\ \ga'(1)=\si(x).
\end{array}
\right.
$$
\noindent By the universal property of the universal cover and Lemma \ref{lift_of_real_paths_to_autom}, given $\ga\in\Ga$, there exists a unique self-diffeomorphism $f_\ga:\Xt\lra\Xt$ such that $f_\ga(\xt) = \ga\car \xt$. This $f_\ga$ covers $\Id_X$ if $\ga(1)=x$ and covers $\si$ if $\ga(1)=\si(x)$. Then $(\Ga,\boxtimes)$ is a group and the map $$\Phi:\begin{array}{ccc}
\Ga &\lra & \piR \\
\ga & \lmt & f_\ga
\end{array} 
$$
\noindent is a group isomorphism which in fact makes the following diagram commute
$$\begin{CD}
1@>>> \piC @>>> \Ga @>\beta>> \Z/2\Z @>>> 1\\
@. @V{\simeq}V{\Phi|_{\piC}}V @V{\simeq}V{\Phi}V @| @.\\
1@>>> \piC @>>> \piR @>\alpha>> \Z/2\Z @>>> 1\\
\end{CD}$$
\noindent where $\beta(\ga) = \Id_X$ if $\ga(1)=x$ and $\beta(\ga)=\si(x)$ if $\ga(1)=\si(x)$.
\end{proposition}

\begin{proof}
It is simplifying to denote by $\si(\ga)$ the (homotopy class of the) path $\si\circ\ga:[0;1]\lra X$. Using the fact that $\si(\ga\ast\ga') = \si(\ga) \ast \si(\ga')$, it is not hard to check that the product $\boxtimes$ on $\Ga$ is associative and that all elements have inverses (the neutral element being the homotopy class of the constant path at $x$). So $\Ga$ is a group.\\
Let $\ga$ and $\ga'$ be two elements in $\Ga$ and denote by $\Phi(\ga)= f_\ga$ and $\Phi(\ga')=f_{\ga'}$. Then $\Phi(\ga\boxtimes\ga')$ is the unique map $h:\Xt\lra\Xt$ covering either $\Id_X$ (if $\ga\boxtimes\ga'$ ends at $x$) or $\si$ (if $\ga\boxtimes\ga'$ ends at $\si(x)$) such that $$h(\xt) = (\ga\boxtimes\ga')\car\xt = \left\{
\begin{array}{cl}
(\ga\ast\ga')\car \xt & \mathrm{if}\ \ga'(1)=x,\\
(\si(\ga)\ast\ga')\car \xt & \mathrm{if}\ \ga'(1)=\si(x).
\end{array}
\right.$$
\noindent Let $\alpha:=\Id_X$ if $\ga'(1)=x$ and $\alpha:=\si$ if $\ga'(1)=\si(x)$, so that $$h(\xt)=(\alpha\circ\ga)\car(\ga'\car\xt) = (\alpha\circ\ga)\car f_{\ga'}(\xt).$$ Then $f_{\ga'}\circ\gat$ is a path starting at $f_{\ga'}(\gat(0))=f_{\ga'}(\xt)$ and satisfying $p\circ(f_{\ga'}\circ\gat)=(p\circ f_{\ga'})\circ \gat=(\alpha\circ p) \circ \gat = \alpha\circ(p\circ\gat) = \alpha\circ \ga$. So $$(\alpha\circ\ga)\car f_{\ga'}(\xt) = (f_{\ga'}\circ\gat)(1) = f_{\ga'}(\gat(1))= f_{\ga'}(\ga\car \xt)$$ and therefore $$h(\xt)=(\alpha\circ\ga)\car f_{\ga'}(\xt) = f_{\ga'}(\ga\car\xt)=f_{\ga'} \big(f_\ga(\xt)\big) = (f_{\ga'}\circ f_\ga)(\xt).$$ Moreover, the map $f_{\ga'}\circ f_\ga$ covers $\si$ if $h$ covers $\si$, and it covers $\Id_X$ if $h$ covers $\Id_X$. So, by uniqueness of such a map, $$\Phi(\ga\boxtimes\ga')=h=f_{\ga'}\circ f_\ga=\Phi(\ga')\circ\Phi(\ga),$$ which proves that $\Phi$ is a group homomorphism from $\Ga$ to $\piR$.\\
To conclude that $\Phi$ is an isomorphism of groups, we observe that the map sending $f\in\piR$ to the homotopy class of $\ga:=p\circ\zeta$, where $\zeta$ is any path from $\xt$ to $f(\xt)$ in $\Xt$, is an inverse for $\Phi$ (this again uses that $\Xt$ is simply connected to show that the given map is well-defined).
\end{proof}

 It will be convenient to refer to a \textit{surjective group homomorphism} such as $$\alpha: \piR\lra \Z/2\Z$$ as an \textit{augmentation} of $\Z/2\Z$. We will write $$\alpha: \piR\lra \Z/2\Z\lra 1$$ to keep in mind that $\alpha$ is a surjective group homomorphism. A homomorphism between two augmentations $\alpha:A\lra\Z/2\Z\lra 1$ and $\beta:B\lra \Z/2\Z\lra 1$ is a group homomorphism $\psi:A\lra B$ which makes the diagram 

$$\begin{CD}
A @>\alpha>> \Z/2\Z @>>> 1 \\
@VV{\psi}V @| \\
B @>\beta>> \Z/2\Z @>>> 1
\end{CD}$$ 

\noindent commute. In particular, $\psi$ sends $\ker\alpha$ to $\ker\beta$. In this language, \textit{the fundamental group of a real curve $(X,\si)$ is an augmentation of the absolute Galois group of $\R$}. A homomorphism between the fundamental groups $\piRp$ and $\piR$ of two real curves $(X',\si')$ and $(X,\si)$ is therefore a group homomorphism $\psi:\piRp\lra \piR$ making the diagram

$$\begin{CD}
1 @>>> \piCp @>>> \piRp @>>> \Z/2\Z @>>> 1\\
@. @VV{\psi|_{\piCp}}V @VV{\psi}V @| @.\\
1 @>>> \piC @>>> \piR @>>> \Z/2\Z @>>> 1\\
\end{CD}$$ 

\noindent commute. Using the description of fundamental groups of real curves given in Corollary \ref{SDP_structure_in_terms_of_paths} and Proposition \ref{description_in_terms_of_paths}, it is fairly straightforward to establish the functorial character of $\pi_1^\R$.

\begin{proposition}
Let $(X',\si')$ and $(X,\si)$ be two real curves and let $\phi:X'\lra X$ be any continuous map such that $\phi\circ\si'=\si\circ\phi$. Choose $x'\in X'$ and set $x=\phi(x')\in X$. Then there is a homomorphism of augmentations $$\begin{CD}
\piRp @>>> \Z/2\Z @>>> 1\\
@VV{\phi_*}V @| \\
\piR @>>> \Z/2\Z @>>> 1
\end{CD}$$ 
 In particular, the fundamental group of a real curve is a topological invariant.
\end{proposition}

\begin{proof}
Due to the condition $\phi\circ\si'=\si\circ \phi$, if $x'\in (X')^{\si'}$, then $x\in X^\si$, so both $\piRp$ and $\piR$ are isomorphic to semi-direct products of the form given in Corollary \ref{SDP_structure_in_terms_of_paths} and $\phi_*$ may be defined by sending $$(\ga,\eps)\in\piCp\rtimes_{\si'}\Z/2\Z$$ to $(\phi(\ga),\eps)\in\piC\rtimes_{\si}\Z/2\Z$. If now $x'\notin (X')^{\si'}$, let us define $\phi_*$ in the following way, using Proposition \ref{description_in_terms_of_paths}~:
\begin{itemize}
\item if $\si(x)\neq x$, send $\ga$ (going either from $x'$ to $x'$, or from $x'$ to $\si'(x')$) to $\phi(\ga)$,
\item if $\si(x)=x$, send $\ga$ to $(\phi(\ga),\Id_X)\in\piC\rtimes_{\si}\Z/2\Z$ if $\ga(1)=x'$ and to $(\phi(\ga),\si)$ if $\ga(1)=\si'(x')$.
\end{itemize}
\noindent In all cases, $\phi_*$ is a homomorphism of augmentations of $\Z/2\Z$.
\end{proof}

\noindent It is somewhat more difficult to define $\phi_*$ without the help of paths but we can nonetheless sketch a proof of functoriality as follows. Given a continuous real homomorphism $\phi:(X',\si')\lra(X,\si)$ and a real cover $q:(Y,\tau)\lra (X,\si)$, there is a real cover $\phi^*q:(\phi^*Y,\phi^*\tau) \lra (X',\si')$ which satisfies $$(\phi^*Y)_{x'} = Y_{\phi(x')} \overset{\phi^*\tau}{\lra} (\phi^*Y)_{\si'(x')}=Y_{\si(\phi(x))}.$$ In particular, any automorphism of the set $(\phi^*Y)_{x'}$ defines an automorphism of the set $Y_{\phi(x')}=Y_x$, compatibly with homomorphism of real covers. Since the group $\piRp$ is the group of automorphisms of the fibre at $x'$ of \textit{all} real covers of $(X',\si')$ and $\piR$ is the group of automorphisms of the fibre at $x=\phi(x')$ of \textit{all} real covers of $(X,\si)$, the construction above defines a map $$\phi_*:\piRp\lra\piR,$$ which is a homomorphism of $\Z/2\Z$-augmentations.

We refer to \cite{Huisman_TLSE} for presentations of fundamental groups of compact connected Klein surfaces and we end this section with two simple examples.

\begin{example}
If $(X,\si)$ is a Klein surface with no real points, then $$\piR\simeq \pi_1^{\mathrm{top}}(X/\si;[x]).$$ Indeed, the canonical projection $q:X\lra X/\si$ is a covering map in this case and $\Xt$ is a universal covering space for $X/\si$. Moreover, a self-map $f:\Xt\lra\Xt$ covers $\Id_{X/\si}$ on $X/\si$ if and only if it covers $\Id_X$ or $\si$ on $X$, which exactly means that  $\pi_1^{\mathrm{top}}(X/\si;[x]) \simeq \piR$. Another way to see it is to say that the pullback map $q^*$ defined by the topological covering $q$ is an equivalence of categories between topological covers of $X/\si$ and real covers of $(X,\si)$, so the Galois groups are isomorphic.
\end{example}

\begin{example}
Let us consider the non-compact Klein surface $$X:=\CP^1\setminus\{s_1;s_2;s_3\},$$ where $s_1,s_2,s_3$ are distinct points in $\RP^1=\Fix(\si)\subset\CP^1$, and let us fix $x\notin\RP^1$ as a base point. We denote by $\eta_1$ the homotopy class of a path from $x$ to $\si(x)$ that passes in between $s_1$ and $s_2$ and does not wrap around any of the punctures. Likewise, $\eta_2$ passes between $s_2$ and $s_3$ and $\eta_3$ passes in between $s_3$ and $s_1$. Each $\eta_i$ is then of order $2$ in $\piR$ and $\zeta_1:=\eta_2 \boxtimes\eta_1$, $\zeta_2:=\eta_3 \boxtimes\eta_2$ and $\zeta_3:=\eta_1 \boxtimes\eta_3$ belong to $\piC$ and satisfy $\zeta_3\ast\zeta_2\ast\zeta_1=1$. It is clear that $\zeta_1,\zeta_2,\zeta_3$ generate $\piC$ and that $\eta_1,\eta_2,\eta_3$ generate $\piR$. More precisely, we have the following presentations $$\piC=<a_1,a_2,a_3\ |\ a_1a_2a_3=1>$$ and $$\piR=<b_1,b_2,b_3\ |\ b_1^2=b_2^2=b_3^2=1>.$$ The augmentation map $\piR\lra\Z/2\Z$ sends each $b_i$ to the non-trivial element of $\Z/2\Z$, so the elements $a_1=b_1b_2$, $a_2=b_2b_3$, $a_3=b_3b_1$ are sent to the trivial element. The sequence $$1\lra <b_1b_2,b_2b_3,b_3b_1> \lra <b_1,b_2,b_3\ |\ b_1^2=b_2^2=b_3^2=1> \lra \Z/2\Z \lra 1$$ is
 the homotopy exact sequence of $\piR$ in this case. The present example generalizes in a straightforward manner to any finite number of punctures $s_1,\cdots,s_n\in\RP^1$. Note that if we do not want the $s_i$ to lie in $\RP^1$, those that do not must come in pairs $(s_i,\si(s_i))$, due to the presence of the real structure on $X$.
\end{example}

\subsection{Galois theory of real covering spaces}\label{Galois_over_R} In this final subsection on the general theory of the fundamental group of a Klein surface, we explain in which sense $\piR$ is a Galois group.

\begin{theorem}[Galois theory of real covering spaces]
Let $(X,\si)$ be a connected Klein surface. The category $\Cov^{\ \R}_{(X,\si)}$ of real covering spaces of $(X,\si)$ is equivalent to the category of discrete left $\piR$-sets. An explicit pair of quasi-inverse functors is provided by the fibre-at-$x$ functor 

$$\Fib_x: \begin{array}{ccc}
Cov^\R_{(X,\si)} & \lra & \piR-\Sets \\
(q: (Y,\tau)\lra (X,\si)) & \lmt & q^{-1}(\{x\}) 
\end{array}
$$

\noindent and the functor

$$Q: \begin{array}{ccc}
 \piR-\Sets & \lra & Cov^\R_{(X,\si)} \\
F & \lmt & (\Xt \times_{\piC} F, \piR/\piC)
\end{array}\, .
$$
\end{theorem}

\begin{proof}
We saw in Proposition \ref{piR_acts_on_fibres} that the fibre at $x\in X$of a real cover $q:(Y,\tau)\lra (X,\si)$ was acted upon by $\piR$ and that homomorphism of real covers defined maps between fibres which are equivariant with respect to this action~: this shows that $\Fib_x$ is indeed a functor from $\Cov^{\ \R}_{(X,\si)}$ to $\piR-\Sets$.\\
Conversely, assume that $F$ is a discrete set endowed with a left action of the group $\piR$. In particular, $\piC\subset \piR$ acts on $\Xt\times F$ and, by Theorem \ref{Galois_theory_of_covers}, $\Xt\times_{\piC} F$ is a cover of $X$. To show that $\Xt\times_{\piC} F$ has a real structure, consider any $h\in\piR$ such that $h$ maps to $(-1)\in\Z/2\Z$ under the augmentation map (this is equivalent to choosing a lifting $\widetilde{\si}:\Xt\lra\Xt$ of $\si:X\lra X$; by the universal property of the universal covering map $p:\Xt\lra X$, a lifting $\widetilde{\si}$ such that $p\circ\widetilde{\si} = \si \circ p$ always exists, but it is not unique, moreover, in general one has $\widetilde{\si}^2\neq\Id_{\Xt}$). Then the map $$\Phi:\begin{array}{ccc} \Xt\times F & \lra & \Xt\times F \\ (y,v) &\lra & (y\cdot h^{-1}, h\cdot v)\end{array}$$ sends $\piC$-orbits to $\piC$-orbits~: for any $f\in\piC$, \begin{eqnarray*}
\Phi(y\cdot f^{-1}, f\cdot v) & = & \big((y\cdot f^{-1})\cdot h^{-1}, h\cdot(f\cdot v)\big)\\
& = & \big((y\cdot h^{-1})\cdot(hfh^{-1})^{-1}, (hfh^{-1})\cdot (h\cdot v)\big)\\
& \sim & (y\cdot h^{-1}, h \cdot v) = \Phi(y,v)
\end{eqnarray*} for $hfh^{-1}\in\piC \lhd \piR$. So $\Phi$ induces a map $$\Xt\times_{\piC} F \lra \Xt\times_{\piC} F,$$ which squares to the identity since $h^2\in\piC$. In other words, the group $$\piR/\piC \simeq \Z/2\Z$$ acts on $Y:=\Xt\times_{\piC} F$, so the latter is a Klein surface $(Y,\tau)$. Since the real structure $\si$ on $X$ is obtained as above by considering $F=\{\mathrm{pt}\}$, the Klein surface $(Y,\tau)$ is indeed a real cover of $(X,\si)$, with fibre $F$. Moreover the real structure $\tau$ on $Y$ does not depend on the choice of $h$, since any two choices differ by an element of $\piC$. This construction is functorial and a quasi-inverse to $\Fib_x$.
\end{proof}

\begin{multicols}{2}

\begin{exercise}
Let $V$ be a vector space and let $\bE$ be an affine space with group of translations $V$. Show that the group of affine transformations of $\bE$ is the semi-direct product $V\rtimes \GL(V)$ where $\GL(V)$ is the group of linear automorphisms of $V$. Recall that the composition law on this semi-direct product is $$(v,A)\cdot(w,B) = (v+Aw, AB).$$
\end{exercise}

\begin{exercise}\label{hol_and_anti_hol_bijections}
Let $X$ be a Riemann surface and let $\Aut^\pm(X)$ be the set of holomorphic and anti-holomorphic bijections from $X$ to $X$.\\ 
\textbf{1.} Show that $\Aut^\pm(X)$ is a group.\\ 
\textbf{2.} Show that if $\si$ is a real structure on $X$ then $\Aut^\pm(X)$ is isomorphic to $\Aut(X)\rtimes \Z/2\Z$ where $\Z/2\Z\simeq\{\Id_X;\si\}$ acts on $\Aut(X)$ by $\si\cdot f=\si\circ f \circ\si^{-1}$.
\end{exercise}

\begin{exercise}
Let $G,H$ be two groups and let $\phi:H\lra \Aut(G)$ be a group homomorphism (a left action of $H$ on $G$ by group automorphisms).\\
\textbf{1.} Show that $\phi$ induces a group homomorphism $\phi^\op: H^\op\lra \Aut(G^\op)^\op$ (a right action of $H^\op$ on $G^\op$ by group automorphisms).\\
\textbf{2.} Show that the map $$\begin{array}{ccc} (G\rtimes H)^\op & \lra & H^\op \ltimes G^\op \\ (g,h) & \lmt & (h,g) \end{array}$$ is a group isomorphism.
\end{exercise}

\begin{exercise}\label{opposite_product_in_pi1}
Consider the usual construction of $\Xt$ as the set of homotopy classes of paths starting at a given $x\in X$ and the action $\eta\curvearrowleft \ga:=\eta\ast\ga$ defined in \eqref{pi1_acts_on_universal_cover} for any $\ga\in\piC$ and any $\eta\in\Xt$. Show that, if $\eta$ is in fact the homotopy class of a \textit{loop} at $x$, the monodromy action $\ga\curvearrowright \eta$ of $\ga$ on $\eta$ as defined in \eqref{monodromy_action} is indeed $\gamma\ast\eta$ (the opposite product of $\eta$ and $\ga$ in $\piC$).
\end{exercise}

\begin{exercise}\label{represent_the_fibre_functor}
Use for instance the book by Szamuely (\cite{Szamuely}) to work out the details of Subsection \ref{reminder_cx_case}. Show in particular that $\Xt$ represents the functor $$\Fib_x:\Cov_X \lra \piC-\Sets.$$ 
\end{exercise}

\begin{exercise}\label{left_and_right_on_Xt}
Show that, if $p:(\Xt,\xt)\lra (X,x)$ is a universal cover and $f\in\Aut(\Xt/X)^{\op}$, we have, for all $y\in\fibre\subset \Xt$, then $f\circ t_y = t_{f(y)}$, where $t_y$ and $t_{f(y)}$ are defined as in Proposition \ref{monodromy_action_bis}.
\end{exercise}

\begin{exercise}
Show that an anti-holomorphic involution of the complex plane $\C$ is conjugate to $\tau_0: z\lmt\ov{z}$ by an element of the group $\Aut(\C) \simeq \{z\lmt az+b : a\in\C^*, b\in\C\}$ of biholomorphic transformations of $\C$.
\end{exercise}

\begin{exercise}
Show that an anti-holomorphic involution of the Poincar\'e upper half-plane $\cH=\{z\in\C\ |\ \Im{z} >0\}$ is conjugate to $\tau_0: z\lmt -\ov{z}$ by an element of the group $\Aut(\cH) \simeq \PSL(2;\R)$ of biholomorphic transformations of $\cH$.
\end{exercise}

\begin{exercise}\label{more_on_SDP_structures} (\textit{Suggested by Andr\'es Jaramillo}). 
Let $(X,\si)$ be a Real Riemann surface and let $\Xt$ be its universal covering Riemann surface.\\ 
\textbf{1.} Show that the existence of a map $\sit:\Xt\lra\Xt$ covering $\sit$ and satisfying $\sit^2=\Id_{\Xt}$ is equivalent to the existence of a splitting of the homotopy exact sequence \eqref{homotopy_exact_sequence}.\\
\textbf{2.} Assume that $\si(x)\neq x$. Show that there exists a splitting of the homotopy exact sequence if and only if there exists a path $\eta$ from $x$ to $\si(x)$ which is of order $2$ in $\piR$, i.e.\ such that $(\si\circ\eta)\ast\eta$ is homotopic to the constant path at $x$.\\
\textbf{3.} Show that this cannot happen if $\Fix(\si)=\emptyset$.
\end{exercise}

\begin{exercise}\label{isomorphic_SDP_structures}
Show that if $x$ and $x'$ lie in the same connected component of $X^\si$, (the homotopy class of) a path from $x$ to $x'$ which is contained in $X^\si$ defines an isomorphism of augmentations
$$\begin{CD}
\piC\rtimes_{\sit}\Z/2\Z @>>> \Z/2\Z \\
@V{\simeq}VV @| \\
\piC\rtimes_{\sit'}\Z/2\Z @>>> \Z/2\Z
\end{CD}
$$ where $\sit$ and $\sit'$ are defined as in Lemma \ref{lift_of_real_structure}, using respectively $x$ and $x'$.
\end{exercise}

\end{multicols}

\section{Representations of the fundamental group of a Klein surface}\label{section_reps_fund_gp}

\subsection{Linear representations of fundamental groups of real curves}

As seen in Section \ref{fund_gp_real_curves}, the fundamental group of a real curve $(X,\si)$ is the group $$\piR=\big\{f: \Xt \lra \Xt\ |\ \exists\,\si_f\in\{\Id_X;\si\},\, p\circ f = \si_f\circ p\big\}^{\op}$$ of self-diffeomorphisms of the universal covering space $\Xt$ which cover either $\Id_X$ or $\si$. The main feature of this group is that it contains the topological fundamental group as a normal subgroup, the kernel of a natural surjective group homomorphism to $\{\Id_X;\si\}^\op\simeq\Z/2\Z$~: \begin{equation}\label{short_exact_sequence} 1 \lra \piC \lra \piR {\lra} \Z/2\Z \lra 1.\end{equation} We recall that a surjective group homomorphism to $\Z/2\Z$ is called an \textit{augmentation} of $\Z/2\Z$ and that $\Xt$ has a unique Riemann surface structure turning the topological covering map $p:\Xt \lra X$ into a holomorphic map. Since $\si:X\lra X$ is anti-holomorphic, a map $f:\Xt\lra\Xt$ satisfying $p\circ f = \si \circ p$ is anti-holomorphic (and of course, if $p\circ f = p$, then $f$ is holomorphic), i.e.\ $\piR$ acts on $\Xt$ by holomorphic and anti-holomorphic transformations, depending on whether $f\in\piR$ covers $\Id_X$ or $\si$ on $X$. As such, it is natural, when studying complex linear representations of $\piR$ to have the elements of $\piC\subset\piR$ act by $\C$-linear transformations of a complex vector space $V$, and those of $\piR\setminus\piC$ act by $\C$-anti-linear transformations of $V$. Such transformations are in particular $\R$-linear and form a group that we denote by $\GL^\pm(V)$ and which contains $\GL(V)$ as an index two subgroup. Explicitly, $\GL^\pm(V)$ is the group of $\C$-linear and $\C$-anti-linear transformations of $V$. What we said about elements of $\piR\setminus\piC$ acting on $V$ by transformations in $\GL^\pm(V)\setminus \GL(V)$ while elements of $\piC$ act by transformations in $\GL(V)$ then precisely amounts to saying that we have a \textit{homomorphism of augmentations} of $\Z/2\Z$, the latter being, by definition, a group homomorphism $\rho: \piR \lra \GL^\pm(V)$ which makes the following diagram commute
\begin{equation}\label{linear_rep}
\begin{CD}
1 @>>> \piC @>>> \piR @>>> \Z/2\Z @>>> 1 \\
@. @VV{\rho|_{\piC}}V @VV{\rho}V @| @.\\
1 @>>> \GL(V) @>>> \GL^\pm(V) @>>> \Z/2Z @>>> 1
\end{CD}
\end{equation}

\noindent where the map $\GL^\pm(V)\lra \Z/2\Z$ takes a $\C$-anti-linear transformation of $V$ to the non-trivial element of $\Z/2\Z$. We will call such a map $\rho:\piR \lra \GL^\pm(V)$ a \textit{linear representation} of $\piR$. If we choose a Hermitian metric $h$ on $V$, we can likewise consider the group $\U^\pm(V,h)$ of $\C$-linear and $\C$-anti-linear isometries of $(V,h)$. It contains the unitary group $\U(V,h)$ as an index two subgroup. A \textit{unitary representation} is then by definition a homomorphism of augmentations
\begin{equation}\label{unitary_rep}
\begin{CD}
1 @>>> \piC @>>> \piR @>>> \Z/2\Z @>>> 1 \\
@. @VV{\rho|_{\piC}}V @VV{\rho}V @| @.\\
1 @>>> \U(V,h) @>>> \U^\pm(V,h) @>>> \Z/2\Z @>>> 1.
\end{CD}
\end{equation}

A \textit{linear real structure} on a complex vector space $V$ is a $\C$-anti-linear involution $\alpha$ of $V$. In particular, $V$ is then the complexification of $V^\alpha$, the real subspace of $V$ consisting of fixed points of $\alpha$ in $V$: $V=V^\alpha\otimes_\R \C$ and we will call the pair $(V,\alpha)$ a \textit{real vector space}. Linear real structures on a complex vector space $V$ always exist (see Exercise \ref{linear_real_structures}). If $V$ has a Hermitian metric $h$, we may assume that $\alpha$ is an (anti-linear) isometry of $h$ and we will say that $\alpha$ and $h$ are \textit{compatible}. Once a real structure has been chosen, the groups $\GL^\pm(V)$ and $\U^\pm(V,h)$ have canonical structures of semi-direct product.

\begin{lemma}\label{psd_structure_for_GL}
Let $(V,\alpha)$ be a real vector space, with a compatible Hermitian metric. Then the group $\GL^\pm(V)$ of $\C$-linear and $\C$-anti-linear transformations of $V$ is isomorphic to the semi-direct product $$\GL(V)\rtimes_\alpha \Z/2\Z$$ where $\{\Id_V;\alpha\}\simeq\Z/2\Z$ acts on $\GL(V)$ by $$\Ad{\alpha}u := \alpha \circ u \circ \alpha^{-1}.$$ Likewise, the group $\U^\pm(V,h)$ of $\C$-linear and $\C$-anti-linear isometries of $(V,h)$ is isomorphic to the semi-direct product $$\U(V) \rtimes_\alpha \Z/2\Z.$$
\end{lemma}

\noindent Note that $\alpha$ is naturally also an element of $\GL^\pm(V)$ (resp. $\U^\pm(V,h)$), which is sent to $\alpha$ under the augmentation map $\GL^\pm(V) \lra \Z/2\Z$, and that if a $\C$-linear transformation $u$ has matrix $A$ in a given basis of $V$ consisting of $\alpha$-fixed vectors (a real basis), then $\alpha\circ u\circ \alpha^{-1}$ has matrix $\ov{A}$ in that basis. The proof of Lemma \ref{psd_structure_for_GL} is proposed as an exercise (see Exercise \ref{proof_psd_structure_for_GL}).

\begin{definition}[Linear and unitary representations]\label{reps_of_the_fund_gp}
A \textbf{linear} (resp. \textbf{unitary}) \textbf{representation} of $\piR$ is a homomorphism of augmentations of $\Z/2\Z$ from $\piR$ to $\GL^\pm(V)$ (resp. $\U^\pm(V,h)$). Explicitly, it is a group homomorphism $$\rho:\piR \lra \GL^\pm(V)$$ (resp. $\rho:\piR \lra \U^\pm(V,h)$) compatible with the augmentation homomorphism $\GL^\pm(V)\lra \Z/2\Z$ (resp. $\U^\pm(V,h) \lra \Z/2\Z$) in the sense of diagram \eqref{linear_rep} (resp. \eqref{unitary_rep}). The set of such representations will be denoted $$\Hom_{\Z/2\Z}(\piR;\GL^\pm(V))$$ (resp. $\Hom_{\Z/2\Z}(\piR;\U^\pm(V,h))$). It has a natural topology, induced by the product topology of the space of all maps from $\piR$ to $\GL^\pm(V)$.
\end{definition}

\noindent Note that, because of the short exact sequence \eqref{short_exact_sequence}, $\piR$ is finitely generated, and that $\GL^\pm(V)$ is a subgroup of the group of $\R$-linear automorphisms of the finite-dimensional, complex vector space $V$, which has a canonical topology, so the topology referred to in Definition \ref{reps_of_the_fund_gp} is the subspace topology inside a product of a finite number of copies of $\GL^\pm(V)$.

The natural question to ask next is the following~: when should two such representations be considered equivalent? A priori, one could conjugate a given map $\rho:\piR \lra \GL^\pm(V)$ by any $u\in \GL^\pm(V)$ without altering the commutativity of diagram \eqref{linear_rep}, since $\GL(V)$ is normal in $\GL^\pm(V)$. We will, however, consider a more restrictive notion of equivalence, where $u$ is taken to lie in $\GL(V)$. This will guarantee the existence of a map from the upcoming representation space $$\HomR/\GL(V)$$ to the usual representation space $$\HomC/\GL(V).$$ As a matter of fact, it also implies that the representation space of $\piR$ only depends on the $\Z/2\Z$-augmentation class of the target group (in our case, $\GL^\pm(V)$), not on its isomorphism class as an extension of $\Z/2\Z$ by $\piC$. 

In what follows, as we want to consider quotient spaces that have a Hausdorff topology, we will restrict our attention to \textit{unitary} representation spaces. The discussion above follows through immediately to the unitary case.

\begin{definition}[Unitary representation variety]\label{rep_var}
Two unitary representations of $\piR$ are called equivalent if they lie in a same orbit of the (continuous) action of $\U(V,h)$ on $$\HomRunit$$ defined by $$u\cdot \rho = \Ad{u}\circ \rho.$$ By definition, the \textbf{unitary representation variety} of $\piR$ into $\U^\pm(V,h)$ is the orbit space $$\HomRunit/\U(V,h),$$ endowed with the quotient topology.
\end{definition}

\noindent Note that we call this representation variety the \textit{unitary} representation variety of $\piR$ even though the target group is $\U^\pm(V,h)$, not $\U(V,h)$. To lighten notation, we will fix once and for all a Hermitian metric $h$ on $V$ and simply denote by $\U(V)$ and $\U^\pm(V)$ the groups $\U(V,h)$ and $\U^\pm(V,h)$. Thanks to our choice of equivalence relation between representations of $\piR$, there is a map from the unitary representation variety of $\piR$ to the unitary representation variety of $\piC$ : \begin{equation*}
\Phi:
\begin{array}{ccc}
 \HomRunits / \U(V) & \lra & \HomCunit / \U(V)\\
\left[\rho\right] & \lmt & \left[\rho|_{\piC}\right] .
 \end{array}
 \end{equation*}
 
\noindent In the next subsection, we study some properties of this map.

\subsection{The real structure of the usual representation variety}

To be able to say something about the map $\Phi$ above, we need to fix an additional structure, namely a linear real structure $\alpha$ on $V$. In particular, the group $\U^\pm(V)$ is now canonically isomorphic to the semi-direct product $\U(V)\rtimes_\alpha\Z/2\Z$, by Lemma \ref{psd_structure_for_GL}. The first step in our study is to understand what the real structure of the unitary representation variety of $\piC$ is. Then, as a second step, we will show that $\Phi$ maps the representation variety of $\piR$ to real points of the representation variety of $\piC$.

The existence of the short exact sequence $$1 \lra \piC \lra \piR \lra \Z/2\Z \lra 1$$ implies the existence of a group homomorphism $$\Z/2\Z \lra \Out(\piC):= \Aut(\piC) / \Inn(\piC)\, .$$ Indeed, let $f\in\piR$ go to $\si\in\Z/2\Z$ under the augmentation map. Then $f$ acts by conjugation on the normal subgroup $\piC$ of $\piR$. Any other choice of an element $f'$ in the fibre at $\si$ differs from $f$ by an element $f^{-1}f'$ of $\piC$, so the associated automorphism of $\piC$ differs from that induced by $f$ by an inner one (indeed, $\mathrm{Ad}_{f}^{-1}\circ\Ad{f'} = \Ad{f^{-1}f'}$ with $f^{-1}f'\in\piC$) and we have a map $\Z/2\Z \lra \Out(\piC)$. This map is indeed a group homomorphism (see Exercises \ref{outer_hom} and \ref{case_of_a_real_point} for more information on this map). Now recall (see Exercise \ref{check_outer_action}) that the group $\Out(\piC)$ acts to the left on the representation variety $$\HomCunit/\U(V)$$ by \begin{equation}\label{outer_action}[\varphi]\cdot [\rho] = [\rho\circ \varphi^{-1}]\end{equation} where $[\varphi]$ is the outer class of $\varphi\in\Aut(\piC)$. This action is called the \textit{outer action}. In our real algebraic context, it yields an action of $\Z/2\Z$ on the representation variety of $\piC$, since we have a map $\Z/2\Z\lra \Out(\piR)$. Explicitly, if $\sit$ is in the fibre above $\si$ of the augmentation map $\piR\to\Z/2\Z$, the corresponding action on $\HomCunit/\U(V)$ is defined by $$[\rho] \lmt [\rho \circ \mathrm{Ad}_{\sit}^{-1}]$$ where $\mathrm{Ad}_{\sit}^{-1}=\Ad{\sit^{-1}}\in\Aut(\piC)$ is conjugation by $\sit^{-1}$ in $\piC$ (we recall that $\sit$ is in general not of order $2$ in $\piR$, so it is important to keep track of the $-1$ exponent above). The resulting element $[\rho\circ\mathrm{Ad}_{\sit}^{-1}]$ of the representation variety does not depend on the choice of $\sit$ in the fibre above $\si$ and, since $\sit^2\in\piC$, we get an action of $\Z/2\Z$ on $\HomCunit/\U(V)$. This involution of the unitary representation variety of $\piC$ is defined in a purely algebraic fashion and makes no use of any additional structure of the target group $\U(V)$. Recall now that the latter has a Galois conjugation $$\Ad{\alpha}: \begin{array}{ccc} \U(V) & \lra & \U(V) \\ u & \lmt & \alpha \circ u \circ \alpha^{-1} \end{array}$$ \textit{coming from the real structure $\alpha$ on $V$} (the composition takes place in $\U^\pm(V)$ but the resulting element indeed lies in $\U(V)$). Then we have an involution 

$$
\begin{array}{ccc}
\HomCunit / \U(V) & \lra & \HomCunit / \U(V) \\
\left[\rho\right] & \lmt & \left[\Ad{\alpha}\circ\rho\right].
\end{array}
$$

\noindent It is not uncommon to denote by $\ov{\rho}:=\Ad{\alpha}\circ\rho$ (indeed, we have seen that complex conjugation is what happens to matrices in an appropriate basis of $V$ when conjugated by the real stucture $\alpha$). Note that fixed points of $\rho\lmt\ov{\rho}$ in particular preserve the real subspace $V^\alpha \subset V$~: in other terms, the representation $\rho$ is $\GL(V^{\alpha})$-valued when $\ov{\rho}=\rho$ (in general, not all fixed points of the induced involution of $\HomCunit/\U(V)$, however, come from such a $\rho$). We now put together the two involutions that we have just considered and obtain a real structure of the representation variety with respect to which representations of $\piR$ behave remarkably (they are real points of this real structure). Note that we call it a real structure because it can be shown that it is anti-holomorphic with respect to the Narasimhan-Seshadri complex structure of the representation variety (see for instance \cite{Sch_JSG}) but that, at this stage, it is just an involution. We denote by $$\cR:= \HomCunit / \U(V) $$ the unitary representation variety of $\piC$, \begin{equation}\label{real_rep_var_psd} \cRsialpha:= \HomRunitspsd / \U(V)\end{equation} the unitary representation variety of $\piR$ and
\begin{equation}\label{real_structure_rep_var}
\kappa: \begin{array}{ccc}
\cR & \lra & \cR \\
\left[\rho\right] & \lmt & \left[\Ad{\alpha}\circ \rho\circ\mathrm{Ad}_{\sit}^{-1}\right]
\end{array}
\end{equation}

\noindent the involution of the representation variety $\cR$ obtaining by combining the previous two. Note that $\kappa$ depends both on $\si$ and $\alpha$.

\begin{theorem}\label{real_points_of_kappa}
The map $$\Phi: \begin{array}{ccc}
\cRsialpha & \lra & \cR \\
\left[\rho\right] & \lmt & \left[\rho|_{\piC}\right]
\end{array}
$$ sends $\cRsialpha$ to points of $\Fix(\kappa)\subset \cR$. It is not surjective in general.
\end{theorem}

\begin{proof}
The involution $\kappa$  takes the class of a group homomorphism $$\chi:\piC \lra \U(V)$$ to the class of $$\Ad{\alpha}\circ\chi\circ\mathrm{Ad}_{\sit}^{-1}: \begin{array}{ccc} \piC & \lra & \U(V)\\ f & \lmt & \alpha\chi(\sit^{-1} f \sit) \alpha^{-1}.\end{array}$$ But, if ${\chi}$ lies in the image of $\Phi$, then $\chi$ extends to a group homomorphism $$\rho:\piR\lra \U(V)\rtimes_{\alpha} \Z/2\Z.$$ So 

\begin{eqnarray*}
\Ad{\alpha}\circ\chi\circ\mathrm{Ad}_{\sit}^{-1} (f) & = & \Ad{\alpha}\circ\rho\circ\mathrm{Ad}_{\sit}^{-1} (f) \\
& = & \alpha\, \big( \rho(\sit^{-1}) \rho(f) \rho(\sit) \big) \,\alpha^{-1} \\
& = & \big(\alpha\rho(\sit)^{-1}\big)\, \rho(f)\, \big(\alpha\rho(\sit)^{-1}\big)\\
& = & \big( \underbrace{\alpha\rho(\sit)^{-1}}_{\in \U(V)}\big) \, \chi(f)\, \big(\underbrace{\alpha\rho(\sit)^{-1}}_{\in \U(V)}\big)^{-1}.
\end{eqnarray*}

\noindent Indeed, $\alpha$ and $\rho(\sit)$ both lie in the fibre above $\si$ inside $\U(V)\rtimes_{\alpha} \Z/2\Z$, so $\alpha\rho(\sit)^{-1}$ lies in the fibre above the trivial element, hence belongs to $\U(V)$. Therefore, $(\Ad{\alpha}\circ\chi\circ\mathrm{Ad}_{\sit}^{-1}) \sim_{\U(V)} \chi$, which means that $[\chi]\in\Fix(\kappa)$. We refer to \cite{Sch_JSG} for the proof that $\mathrm{Im}\,\Phi\neq\Fix(\kappa)$ in general.
\end{proof}

\begin{multicols}{2}

\begin{exercise}\label{linear_real_structures}
Show that the set of (linear) real structures on a complex vector space $V$ is in bijection with the homogeneous space $\GL(n;\C)/\GL(n;\R)$, where $n$ is the complex dimension of $V$.
\end{exercise}

\begin{exercise}\label{proof_psd_structure_for_GL}
Prove Lemma \ref{psd_structure_for_GL}. Prove that if a $\C$-linear transformation $u$ has matrix $A$ in a given real basis of $(V,\alpha)$, then $\alpha\circ u\circ\alpha^{-1}$ has matrix $\ov{A}$, the complex conjugate of $A$.
\end{exercise}

\begin{exercise}\label{outer_hom}
Let $G$ be a group. Consider $g\in G$ and $f\in\Aut(G)$. Show that $f\circ\Ad{g}\circ f^{-1} = \Ad{f(g)}$ and that if there is a short exact sequence $$1 \lra G \lra H \lra I \lra 1$$ then there is a group homomorphism $$I\lra \Out(G):=\Aut(G)/\Inn(G)$$ where $$\Inn(G) := \{\Ad{g} : g\in G\}$$ is the group of inner automorphisms of $G$.
\end{exercise}

\begin{exercise}\label{check_outer_action}
Check that the outer action \eqref{outer_action} is well-defined.
\end{exercise}

\begin{exercise}\label{case_of_a_real_point}
\textbf{a.} Show that, if $\si(x)=x$, then the homomorphism $$\Z/2\Z \lra \Out(\piC)$$ sends $\si\in\Z/2\Z$ to the group auto\-morphism $$\si_*:\piC \lra \piC$$ induced by $\si:X\lra X$.\\ \textbf{b.} Show that, in this case, $\Z/2\Z$ in fact already acts to the left on $\HomCunit$, the action being given by $\rho\lmt\rho\circ\si_*^{-1}$.
\end{exercise}

\begin{exercise}
Let $p:B\lra C$ be a surjective group homomorphism and denote by \begin{equation}\label{ses} 1 \lra A \lra B \lra C \lra 1\end{equation} the associated short exact sequence of groups.\\
\textbf{1.} Let $s:C\lra B$ be a group homomorphism satisfying $p\circ s=\Id_C$ (such a map is called a \textit{splitting} of the short exact sequence \eqref{ses}) and consider the associated semi-direct product $A\rtimes_s C$, where $$(a,c)(a',c') = (as(c)a's(c)^{-1},cc').$$ Show that the map $$\phi_s:\begin{array}{ccc} A \rtimes_s C & \lra & B \\ (a,c) & \lmt & a\, s(c) \end{array}$$ induces an isomorphism of extensions of $C$ by $A$~:
$$\begin{CD}
A @>>> A\rtimes_s C @>>> C \\
@| @VV{\phi_s}V @| \\
A @>>> B @>>> C. 
\end{CD}
$$
\textbf{2.} Let $u\in A$. Show that $s':=\Ad{u}\circ s$ is another splitting of \eqref{ses} and that $$A\rtimes_{\Ad{u}\circ s} C \simeq A\rtimes_s C$$ as augmentations of $C$~:
$$\begin{CD}
A @>>> B @>>> C \\
@VV{\Ad{u}}V @VV{\Ad{u}}V @| \\
A @>>> B @>>> C.
\end{CD}
$$
\end{exercise}

\end{multicols}


\end{document}